\documentclass[dvips,a4paper,11pt]{amsart}

\usepackage{multirow}%
\usepackage{color}%
\usepackage{amsmath}%
\usepackage{amsfonts}%
\usepackage{amssymb}%
\usepackage{amsfonts}
\usepackage{graphicx}
\usepackage{dsfont}
\usepackage{titletoc}
\usepackage{fullpage}
\usepackage[runin]{abstract}
\usepackage[round]{natbib}

\numberwithin{subsection}{section}
\newtheorem{theorem}{Theorem}[section]
\theoremstyle{plain}

\newtheorem{assumption}{Assumption}

\newtheorem{corollary}[theorem]{Corollary}

\newtheorem{definition}{Definition}

\newtheorem{lemma}[theorem]{Lemma}

\newtheorem{proposition}[theorem]{Proposition}
\newtheorem{remark}{Remark}

\numberwithin{equation}{section}

\newcommand{\indep}{\perp\hspace{-.25cm}\perp}
\newcommand{\E}{\mathbb{E}}
\renewcommand{\P}{\mathbb{P}}
\newcommand{\R}{\mathbb{R}}
\newcommand{\N}{\mathbb{N}}

\newcommand{\matdeux}[4]{ \begin{pmatrix} #1 & #2 \\ #3 & #4 \end{pmatrix} }
\renewcommand{\^}[1]{\widehat{#1}}
\newtheorem*{theth}{Theorem}

\newcommand{\dd}{ \mathrm{d} }
\newcommand{\parti}[1]{\frac{\partial}{\partial #1} }
\DeclareMathOperator{\im}{Im} 
\DeclareMathOperator{\tr}{tr} 
\DeclareMathOperator{\diag}{diag} 
\DeclareMathOperator{\spann}{span} 
\DeclareMathOperator{\var}{var} 
\DeclareMathOperator{\rank}{rank}

\begin{document}
\bibliographystyle{plainnat}

\title[Test Function]{Test function: A new approach for covering the central subspace}
\author{Fran\c cois Portier and Bernard Delyon}

\maketitle
\begin{center}
\textit{{\small IRMAR, University of Rennes 1\\ Campus de Beaulieu\\ 35042 Rennes Cedex, France\\
E-mails: \emph{\textbf{francois.portier@univ-rennes1.fr\\ bernard.delyon@univ-rennes1.fr }}}}
\end{center}

\bigskip

\begin{abstract}
\noindent In this paper we offer a complete methodology for sufficient dimension reduction called the test function (TF). TF provides a new family of methods for the estimation of the central subspace (CS) based on the introduction of a nonlinear transformation of the response. Theoretical background of TF is developed under weaker conditions than the existing methods. By considering order $1$ and $2$ conditional moments of the predictor given the response, we divide TF in two classes. In each class we provide conditions that guarantee an exhaustive estimation of the CS. Besides, the optimal members are calculated via the minimization of the asymptotic mean squared error deriving from the distance between the CS and its estimate. This leads us to two plug-in methods which are evaluated with several simulations. 

\bigskip

\noindent \textbf{AMS 2000 subject classifications:} Primary 62G08; secondary 62H11, 62H05.

\noindent \textbf{Key words:} Sufficient dimension reduction; Central subspace; Inverse regression; Slicing estimation. 
\end{abstract}

\section{Introduction}

Dimension reduction in regression aims at improving poor convergence rates derived from the nonparametric estimation of the regression function in large dimension. It attempts to provide methods that challenge the curse of dimensionality by reducing the number of predictors. A specific dimension reduction framework, called the \textit{sufficient dimension reduction} (SDR) has drawn attention in the last few years. Let $Y$ be a random variable and $X$ a \textit{p}-dimensional random vector. To reduce the number of predictors, it is proposed to replace $X=(X_1,...,X_p)^T$ by a number smaller than $p$ of linear combinations of the predictors. The new covariate vector has the form $PX$, where $P$ can be chosen as an orthogonal projection on a subspace $E$ of $\R^p$. Clearly, this kind of methods relies on an alchemy between the dimension of $E$, which needs to be as small as possible, and the preservation of the information carried by $X$ about $Y$ through the projection on $E$. In the SDR literature, mainly two kind of spaces have been studied. First a \textit{dimension reduction subspace} (DRS) [\citet*{li1991}] is defined by the conditional independence property
\begin{equation}\label{drs}
Y\indep X \ |\ P_cX,
\end{equation}
where $P_c$ is the orthogonal projection on a DRS. With words, it means that knowing $P_cX$, there is no more information carried by $X$ about $Y$. It is possible to show that (\ref{drs}) is equivalent to
\begin{equation}\label{drs2}
\mathbb{P}(Y\in A|X)=\mathbb{P}(Y\in A|P_{c}X),
\end{equation}
for any measurable set $A$. Moreover under some additional conditions [\citet*{cook1998}], the intersection of all the DRS is itself a DRS. Consequently, there exists a unique DRS with minimal dimension and we call it the \textit{central subspace} (CS). In this article the CS is noted $E_c$. Secondly, another space called a \textit{mean dimension reduction subspace} (MDRS) has been defined in \citet*{cook2002} as
\begin{equation}\label{mdrs2}
\E[Y|X]=\E[Y|P_mX],
\end{equation}
where $P_m$ is the orthogonal projection on a MDRS. Clearly, the existence of a MDRS requires a weaker assumption than the existence of a DRS and therefore it seems to be more appropriate to the context of regression. Because of the analogous equation of (\ref{mdrs2}),
\begin{equation}\label{mdrs}
Y\indep \E[Y|X] \ \ |\ P_mX,
\end{equation}
the definition of a MDRS imposes that all the dependence between $Y$ and its regression function on $X$ is carried by $P_mX$. If the intersection of all the MDRS is itself a MDRS, then it is called the \textit{central mean subspace} (CMS) [\citet*{cook2002}]. In the following the CMS is noted $E_m$. Finally, notice that because a DRS is a MDRS, the CS contains the CMS.\newline
There exists many methods for estimating the CS and the CMS and these methods can be divided into two groups, those who require some assumptions on the distribution of the covariates and those who do not. The second group includes \textit{structure adaptive method} (SAM) [\citet*{hristache2001}], \textit{minimum average variance estimation} (MAVE) [\citet*{xia2002}], and \textit{structural adaptation via maximum minimization} (SAMM) [\citet*{dalalyan2008}]. Those methods are free from conditions on the predictors but require a non parametric estimation of the regression function $\E[Z|X=x]$. In this article we are concerned only with methods of the first group and we quote them in the following.

\bigskip
To be more comprehensive, from now on we work in term of standardized covariate $Z=\Sigma^{-\frac 1 2}(X-\E[X])$ with $\Sigma=\var(X)$. Hence we define the standardized CS as $\Sigma^{\frac 1 2} E_c$. Since there is no ambiguity, we still note it $E_c$ and we still denote by $P_c$ the orthogonal projection on it. For any matrix $M$, we note $\spann(M)$ the space generated by the columns of $M$. 

All the methods of the first group derive from the principle of inverse regression : instead of studying the regression curve which implies high dimensional estimation problems, the study is based on the inverse regression curve $\E[Z|Y=y]$ or the inverse variance curve $\var(Z|Y=y)$. To infer about the CS, order $1$ moment based methods require that
\begin{assumption}\label{linearity}{(Lineariy condition)}
\begin{equation*}
Q_c \mathbb{E}[Z|P_cZ]=0\quad\text{a.s.,}
\end{equation*}
\end{assumption}
\noindent where $Q_c=I-P_c$. Under the linearity condition and the existence of the CS, it follows that $\E[Z|Y] \in E_c$ a.s. and then if we divide the range of $Y$ into $H$ slices $I(h)$, we have for every $h$, 
\begin{equation}\label{irc}
m_h=\E[Z|Y\in I(h)]\in E_c,
\end{equation}
and clearly, the space generated by some estimators of the $m_h$'s estimate the CS, or a subspace of it. To obtain a basis of this subspace, \citet*{li1991} proposed a principal component analysis and this leads to an eigendecomposition of the matrix
\begin{equation}\label{uut}
\widetilde{M}_{SIR} = \sum_h p_h m_h m_h^T,
\end{equation}
where $p_h=\P(Y\in I(h))$. Many methods relying on the inverse regression curve such as \textit{sliced inverse regression} (SIR) [\citet*{li1991}] have been developed. Other ways to estimate the inverse regression curve are investigated in \textit{kernel inverse regression} (KIR) [\citet*{zhu1996}] and \textit{parametric inverse regression} (PIR) [\citet*{bura1997}]. Instead of a principal component analysis, the minimization of a discrepancy function is studied in \textit{inverse regression estimator} (IRE) [\citet*{cook2005}] to obtain a basis of the CS. For a complete background about order $1$ methods, we refer to \citet*{cook2005}.

Otherwise, in addition to the linearity condition order $2$ moments based methods require that
\begin{assumption}{(Constant conditional variance (CCV))}
\begin{equation*}
\var(Z|P_cZ)= Q_c\quad\text{a.s.,}
\end{equation*}
\end{assumption}
\noindent then under the linearity condition, CCV and the existence of the CS, it follows that $\spann(var(Z|Y)-I)\in E_c$ a.s. and by considering a slicing of the response, we have 
\begin{equation}\label{ivc}
\spann(v_h-I)\subset E_c,
\end{equation} 
where $v_h=\var(Z|Y\in I(h))$. Since the spaces generated by the matrices $(v_h-I)$'s are included in the CS, \textit{sliced average variance estimation} (SAVE) in \citet*{cook1991} proposed to make an eigendecomposition of the matrix
\begin{equation*}
\widetilde{M}_{SAVE}=\sum_h p_h(v_h-I)^2,
\end{equation*}
to derive a basis of the CS. Another combination of matrices based on the inverse variance curve is developed in \textit{sliced inverse regression-II} (SIR-II) [\citet*{li1991}]. More recently, \textit{contour regression} (CR) [\citet*{li2005}], and \textit{directional regression} (DR) [\citet*{li2007}] investigate a new kind of estimator based on empirical directions. Besides, methods for estimating the CMS also require Assumptions $1$ and $2$. They include \textit{principal Hessian direction} (pHd) [\citet*{li1992}], and \textit{iterative Hessian transformation} (IHT) [\citet*{cook2002}]. To clear the failure of certain methods when facing pathological models and keep their efficiency in other cases, some combinations of the previous methods as SIR and SIR-II, SIR and pHd or SIR and SAVE have been studied in \citet*{saracco2003} and \citet*{ye2003}.

As we have just highlighted, Assumptions $1$ and $2$ are necessary to respectively characterize the CS with the inverse regression curve and the inverse variance curve. A first point is that the linearity condition and CCV assumed together are really close to an assumption of normality on the predictors. Moreover for each quoted method, these assumptions guarantee only that the estimated CS is included asymptotically in the true CS. A crucial point in SDR literature and a recent new challenge is to propose some methods that allow an \textit{exhaustive} estimation of the CS under mild conditions. Some recent research are concerned with this problem, \citet*{li2005} and \citet*{li2007} proposed a new kind of assumptions that guarantee the exhaustivity.   
\bigskip

There exists a large range of methods aiming at the estimation of the CS. In this paper, we try to propose a general point of view about SDR by introducing the test function method (TF). The original basic idea of TF is to investigate the dependence between $Z$ and $Y$ by introducing nonlinear transformations of $Y$, and inferring about the CS through their covariances with $Z$ or $ZZ^T$. Actually, an important difference between TF and other methods is that neither the inverse regression curve and nor the inverse variance curve are estimated as it is suggested by equation (\ref{irc}) and (\ref{ivc}). In this paper, these two curves are a working tool but the inference about the CS is obtained through some covariances. More precisely, the CS is obtained either by inspection of the range of  
\begin{equation*}
\E[Z\psi(Y)],
\end{equation*}
when $\psi$ varies in a well chosen finite family of function or either by an eigendecomposition of 
\begin{equation*}
\E[ZZ^T\psi(Y)],
\end{equation*}
where $\psi$ is a well chosen function. Hence two kind of methods can be distinguished, the order $1$ test function methods (TF1) and the order $2$ test function methods (TF2). Notice that $\widetilde{M}_{SIR}$ is an estimate of $\E[Z\E[Z|Y]^T]$, hence SIR may be seen as a particular case of TF1. In this paper, we show that TF allows to relax some hypotheses commonly assumed in the literature, especially we alleviate the CCV hypothesis for TF2. Moreover for each methods, we provide mild conditions ensuring an exhaustive characterization of the CS. Finally, an asymptotic variance analysis leads us to the optimal transformation of $Y$ for the estimation of the CS. As a result a significant improvement in accuracy is targeted by TF. The present work is divided in the three following principal parts :   
\begin{itemize}
\item[\textbullet]  Existence of the CS

\item [\textbullet]  Exhaustivity of TF

\item [\textbullet]  Optimality for TF
\end{itemize}
More precisely, it is organized as follows. In section \ref{s2}, we investigate some new conditions ensuring the existence of the CS and the CMS. In section \ref{s3}, we introduce TF1 and TF2 by providing some basic results. Conditions for an exhaustive characterization of the CS are presented in section \ref{s4}. The choice of the optimal transformation of the response for TF1 and TF2 is detailed in section \ref{s5}. Accordingly, we propose two plug-in methods deriving from the minimization of the MSE. And finally, in section \ref{s7} we compare both methods to existing ones through simulations.

\section{Unicity of the central subspace and the central mean subspace}\label{s2}

Conditions on the unicity of subspaces that allow a dimension reduction are investigated in this section. This problem has drawn the attention early in the literature but it seems not to be the case anymore. As a consequence of the definition of the CS (resp. CMS), its existence is equivalent to the unicity of a DRS (resp. MDRS) with minimal dimension. In \citet*{cook1998}, proposition $6.4$ p.$108$, it is shown that the existence of the CS can be obtained by constraining the distribution of $X$. More precisely, the CS exists under the assumption that $X$ has a convex density support. Moreover, in \citet*{cook2002}, the existence of the CMS is ensured under the same condition than the CS. We prove in Theorem \ref{CMS} and Corollary \ref{CS} below that the convexity assumption can be significantly weakened. Here, the standardization of the predictors do not change the presentation of our results, hence we present it for $X$. For a comprehensive proof of our theorems we need the following lemma.
\begin{lemma}\label{lem1}
If the restriction of $X$ to the ball of $\mathbb{R}^p$ with radius $r$ and center $x_0$ has a strictly positive density, then the intersection of all the MDRS is a MDRS on this ball, i.e.
\begin{equation*}
(\mathbb{E}[Y|X]-\mathbb{E}[Y|RX])\mathds{1}_{\{X\in B(x_0,r)\}}=0,
\end{equation*}
where $R$ denotes the orthogonal projection onto the intersection of all MDRS.
\end{lemma}  
\begin{proof}
It suffices to show the theorem for two MDRS. We first make the proof for a ball centered in $0$, and then we apply it to $X-x_0$. Let $E$ and $E'$ be two MDRS and $P$ and $P'$ their respective orthogonal projections. Denote by $R$ the orthogonal projection onto the subspace $E\bigcap E'$. Using the definition of a MDRS, 
\begin{equation*}
\mathbb{E}[Y|X]=\mathbb{E}[Y|PX]=\mathbb{E}[Y|P'X]\quad \text{a.s..}
\end{equation*}
Let $g(PX)$ and $h(P'X)$ denote the last two functions of the preceding equation. Using that $X$ has a strictly positive density on the unit sphere, we can write
\begin{equation}\label{ae}
g(Px)=h(P'x)\quad\text{a.e. on } B(0,r).
\end{equation}
Let $\varepsilon >0 $, and $\varphi_k$ be a unit approximation with compact support $B(0,\varepsilon)$, we define the function $f_k :B(0,r)\rightarrow \R$ such that
\begin{equation*}
f_k(x)=(g\circ P)\ast \varphi_k\ (x).
\end{equation*}
Then, we have for all $ x$,
\begin{eqnarray*}
f_k(x) &=& \int g(P(x-y))\varphi_k(y) \dd y\\
 &=& f_k(Px).
\end{eqnarray*}
Moreover, for all $x\in B(0,r-\varepsilon)$ since in the above integral $x-y\in B(0,r)$, using (\ref{ae}) we derive
\begin{equation*}
f_k(x)=(h\circ P')\ast \varphi_k\ (x),
\end{equation*}
and similarly we obtain $f_k(x)=f_k(P'x)$. Since $f_k(x)=f_k(Px)=f_k(P'x)$, a simple iteration process provides for all $x\in B(0,r-\varepsilon)$,
\begin{equation*}
f_k(x)=f_k((PP')^nx).
\end{equation*}
Since $f_k$ is a continuous function and $(PP')^n\underset{n\rightarrow \infty}{\longrightarrow} R$,
\begin{equation*}
f_k(x)=f_k(Rx),\quad x\in B(0,r-\varepsilon).
\end{equation*}
To conclude, the unit approximation theorem gives us the convergence
\begin{equation*}
f_k\circ R \overset{L_1}{\longrightarrow} g\circ P.
\end{equation*}
Thus, from $f_k(RX)$ we can derive a subsequence $f_{n_k}(RX)$ that converge almost surely to $g(PX)$, proving that $\E[Y|X]$ is a function of $RX$. This completes the first part of the proof.

Now suppose that $X$ has a strictly positive density onto the ball of radius $r$ and center $x_0$. Define $\widetilde{X}=X-x_0$, it is clear that a MDRS for $X$ is also a MDRS for $\widetilde{X}$ and conversely. Then, since $\widetilde{X}$ is centered in $0$, the intersection of two MDRS is still a MDRS for $\widetilde{X}$ and obviously for $X$.
\end{proof}

The following theorem provides us the existence of the CMS under a weaker condition than in~\citet*{cook1998}. The same result on the existence of the CS is presented in a corollary that follows the theorem.
\begin{theorem}\label{CMS}
If $X$ has a density such that the Lebesgue measure of the boundary of its support is equal to $0$, then the CMS exists.
\end{theorem}
\begin{proof}
Denote by $F\subset \R^p$ the support of the density of $X$. A first step consists in showing that its interior $\mathring{F}$ can be covered by a countable number of balls included in $\mathring{F}$. Secondly, we apply Lemma \ref{lem1} to each of this balls to obtain that the intersection of two MDRS on $\mathring{F}$ is a MDRS on $\mathring{F}$. Finally, the unicity is shown.\newline
Let $x\in \mathring{F}$, then it exists $r>0$ such that $B(x,r)\subset\mathring{F}$. It is possible to find a ball with rational center and radius included in $B(x,r)$ and containing $x$. Thus any $x$ of $\mathring{F}$ is contained in a ball with rational center and radius included in $\mathring{F}$. In other words, the set A formed by all the balls $B(q,r_0)\subset \mathring{F}$ with $q$ and $r_0$ rational covers $\mathring{F}$. Therefore, by applying Lemma \ref{lem1}, we have for all $B(q,r_0)\in A$,
\begin{equation*}
|\mathbb{E}[Y|X]-\mathbb{E}[Y|RX]|\mathds{1}_{\{X\in B(q,r_0)\}}=0,
\end{equation*}
Since $A$ is a countable set,
\begin{equation*}
\sum_{(q,r_0)\in A}|\mathbb{E}[Y|X]-\mathbb{E}[Y|RX]|\mathds{1}_{\{X\in B(q,r_0)\}}=0,
\end{equation*}
then,
\begin{equation*}
|\mathbb{E}[Y|X]-\mathbb{E}[Y|RX]|\sum_{(q,r_0)\in A}\mathds{1}_{\{X\in B(q,r_0)\}}=0. 
\end{equation*}
By assumption $\P(X\in \mathring{F})=1$, then the right-hand side is almost surely strictly positive, and thus
\begin{equation*}
\mathbb{E}[Y|X]=\mathbb{E}[Y|RX] \quad\text{a.s..}
\end{equation*}
Consequently, the intersection of two MDRS is a MDRS. To complete the proof, all the MDRS with minimal dimension have the same dimension and their intersection is still a MDRS with minimal dimension. Hence a MDRS with minimal dimension is unique and the CS exists.
\end{proof}

\begin{corollary}\label{CS}
If $X$ has a density such that the Lebesgue measure of the boundary of its support is equal to $0$, then the CS exists.
\end{corollary}
\begin{proof}
Supposed it exist two different DRS with minimal dimension. By equations (\ref{drs2}) and (\ref{mdrs2}), these DRS are MDRS for the random variables $\mathds{1}_{Y\in A}$ and $X$, for any measurable set $A$. Because we can apply Theorem \ref{CMS}, it is impossible.

\end{proof}

\section{Test function methodology and assumptions}\label{s3}
In the previous section, we focused on conditions that guarantee the existence of the CS and the CMS under the respective Assumptions (\ref{drs}) and (\ref{mdrs}). Since TF is only concerned about the CS estimation, we assume from now on that $X$ satisfies the condition of Corollary \ref{CMS}. As it is detailed in the introduction the estimation of the CS raised two kind of conditions. Those that guarantee a characterization of the CS, and those that permit to cover the entire subspace. In this section we are concerned about the first one. Moreover, we explain our next results in a simple way using the standardized covariates. We denote by $d$ the dimension of $E_c$.

\subsection{Order $1$ test function.}\label{basic}
Model (\ref{drs}) implies that all the information detained by $Z$ about $Y$ is carried by $P_cZ$. To find $E_c$, as pointed out by ~\citet*{li1991} and explained in many articles on the subject, a natural idea is to focus on the inverse regression curve $\mathbb{E}[Z|Y]$. Actually if (\ref{drs}) holds, we can write the inverse regression curve as $\mathbb{E}[\mathbb{E}[Z|P_cZ]|Y]$. Clearly, the linearity condition implies that $\E[Z|P_cZ]\in E_c$ and then $\E[Z|Y]$ is with probability $1$ a vector of $E_c$. To our knowledge, all the order $1$ methods target an estimation of the subspace drawn by $\E[Z|Y]$. As described in the introduction, TF1 does not rely on this idea but also requires the linearity condition. Let us have few words about this assumption.
\begin{remark}
The linearity condition is often equated with an assumption of sphericity on the distribution of the predictor. This is well known that if $Z$ is spherical then it satisfies the linearity condition but the converse is false. Actually, linearity condition and sphericity are not so closely related: in \citet*{eaton1986}, it is shown that a random variable $Z$ is spherical if and only if $\mathbb{E}[QZ|PZ]=0$ for every rank $1$ projection $P$ and $Q=I-P$. Clearly, at this stage, the sphericity seems to be a too large restriction to obtain the linearity condition. However unlike the sphericity, since we do not know $P_c$ the linearity condition could not be checked. An assumption closely related to the linearity condition is to ask the distribution of $Z$ to be invariant by the orthogonal symmetry to the space $E_c$, i.e. $Z\overset{d}{=} (2P_c-I)Z$. Then for any measurable function $f$,
\begin{eqnarray*}
\mathbb{E}[Q_cZ f(P_cZ)] = -\mathbb{E}[Q_cZ f(P_cZ)],
\end{eqnarray*}   
which implies the linearity condition. Recalling that sphericity means invariance in distribution by every orthogonal transformation, we have just shown that an invariance in distribution by a particular one suffices to get the linearity condition. Moreover, the assumption of sphericity suffers from the fact that if we add to $Z$ some independent components then, the resulting vector is no longer spherical whereas the linearity condition is still satisfied. 
\end{remark}

A way to introduce TF1 is to consider some relevant facts of the SIR estimation. As explained in the introduction, SIR consists in estimating the matrix 
\begin{equation*}
M_{SIR}=\mathbb{E}\left[Z\mathbb{E}[Z|Y]^T\right].
\end{equation*}
which column space is included in the CS. To make that possible, a slicing approximation of the conditional expectation $\mathbb{E}[Z|Y]$ is conducted and it leads to $\widetilde{M}_{SIR}$ of equation (\ref{uut}). Because $p_h>0$, it is clear that
\begin{equation}\label{span}
\spann(\widetilde{M}_{SIR})=\spann\left(\E[Z\mathds{1}_{\{Y\in I(h)\}}],\ h=1,...,H\right),
\end{equation}
and it follows that SIR estimates a subspace spanned by the covariances between $Z$ and a family of $Y$-measurable functions. The first goal of TF1 is to extend SIR to other families of functions $\Psi_H$ by estimating $E_c$ through $\spann\left(\E[Z\psi(Y)],\ \psi\in \Psi_H\right)$. Moreover, notice that
\begin{equation}\label{msirtilde}
\widetilde{M}_{SIR}=\E\left[Z \left(\psi_{1}(Y),...,\psi_{p}(Y)\right)\right],
\end{equation}
where $\psi_{k}(y) = \sum_h \alpha_{k,h}\mathds{1}_{\{y\in I(h)\}}$ and $\alpha_{k,h}=\mathbb{E}[Z_k |Y\in I(h)]$. It follows from (\ref{msirtilde}) and (\ref{span}) that
\begin{equation*}
\spann\left(\E[Z\mathds{1}_{\{Y\in I(h)\}}],\ h=1,...,H\right)=\spann\left(\E[Z \psi_{k}(Y)],\ k=1,...,p\right),
\end{equation*}
and clearly SIR synthesizes the information contains in a subspace generated by $H$ vectors into one generated by $p$ vectors. Although each of these spaces are equal, it is not the case for their respective estimators. Accordingly, another issue for TF1 is to choose the $p$ functions $\psi_{k}$ in order to minimize the variance of the estimator. 

The following theorem is not really new. Yet, it makes a simple link between TF1 and the CS. An important fact is that Theorem \ref{th1} provides a vector in $E_c$ for every measurable function.  

\begin{theorem}\label{th1}
Assume that $Z$ satisfies Assumption \ref{linearity} and has a finite first moment. Then, for every measurable function $\psi:\mathbb{R}\rightarrow \mathbb{R}$ such that $\E[Z\psi(Y)]<\infty$, we have
\begin{equation*}
\mathbb{E}[Z\psi(Y)]\in E_c.
\end{equation*}
\end{theorem}

\begin{proof}
Thanks to the existence of the CS, $\mathbb{E}[Z\psi(Y)]=\mathbb{E}\left[\mathbb{E}[Z|P_cZ]\psi(Y)\right]$, and thanks to the linearity condition, $Q_c\mathbb{E}[Z\psi(Y)] =0$.
\end{proof}

\subsection{Order 2 test function.}
TF2 relies exactly on the same approach than TF1 with the difference that it involves higher conditional moments of $Z$ knowing $Y$. Indeed, we are interested in the space generated by the columns of the matrix $\E[ZZ^T\psi(Y)]$ where $\psi$ denote a measurable function. The same issues are addressed : many functions $\psi$ are considered in a first time, and then we look for an optimal function.

Let us recall a known fact often presented as the SIR pathology. Consider the regression model
\begin{equation}\label{patho}
Y=g(Z_1,Z_2,\varepsilon),
\end{equation}
where $\varepsilon\indep Z\in \R^p$ and $g$ is symmetric with respect to its first coordinate. Assume also that $(Z_1,Z_2)\overset{\text{d}}{=} (-Z_1,Z_2)$. Then thanks to the linearity condition we have $ Q_c \E[Z\psi(Y)]=0$ whereas the previous assumptions clearly imply that $\E[Z_1\psi(Y)]=\E[-Z_1\psi(Y)]$. Therefore for any measurable function $\psi$, we have that  $\E[Z\psi(Y)]=\E[(0,Z_2,0,...,0)^T\psi(Y)]$ and consequently the first direction $(1,0,...,0)^T$ cannot be reached by any method based on the inverse regression curve. Clearly, TF1 is sensitive to the SIR pathology. Facing this difficulty an idea developed first in~\citet*{li1991} and ~\citet*{cook1991} is to explore some higher conditional moments of $Z$ given $Y$. Thus methods as SIR-II, SAVE, CR, or DR are interested in some properties of the matrix $\mathbb{E}[ZZ^T|Y]$. It is also the case for TF2. Nevertheless we do not follow the same path specially concerning the assumptions required to explore this second order moment. These kind of method assume first that $Z$ has an elliptical distribution or at least satisfies the linearity condition, and secondly that $\var(Z|P_cZ)$ is a constant, i.e. CCV. The following proposition shows how strong are the last two assumptions.
\begin{proposition}\label{bryc}
Let $Z$ be a random vector of $\mathbb{R}^p$ ($p\geq 2$) with a finite second order moment. If $Z$ is spherical and if $\var(Z|PZ)=const.$ for some orthogonal projection $P$ , then $Z$ is normal and conversely. 
\end{proposition}
\begin{proof}
This proposition follows from Theorem 4.1.4, p.48 of \citet*{bryc1995}.  
\end{proof} 

Accordingly, assumptions required for order $2$ methods are realy close to the assumption of normality on the distribution of the predictors. TF2 works under weaker conditions. Actually, the CCV condition is no longer needed and we substitute it by the following assumption. 

\begin{assumption}{(Diagonal conditional variance (DCV))}\label{variance}
\begin{equation*}
\var(Z|P_cZ)=\lambda^{\ast}_{\omega} Q_c\quad \text{a.s.,}
\end{equation*}
with $\lambda^{\ast}_{\omega}$ a real random variable.
\end{assumption}
In Remark \ref{ccvdcv} we attempt to compare CCV and DCV. To facilitate futures proofs and for a better understanding of such a condition we provide an equivalent form in the following lemma. 
\begin{lemma}\label{lemma}
Assume that $Z$ has a finite second moment. Then the following assertions are equivalent,
\begin{enumerate}
\itemsep 8pt
\item \label{e1} for any orthogonal transformation $H$ such that $HP_c=P_c$, we have
\begin{equation*}
\var(Z|P_cZ)=\var(HZ|P_cZ),
\end{equation*}  
\item \label{e2} $\var(Z|P_cZ)=\lambda^{\ast}_{\omega}Q_c$ with $\lambda^{\ast}_{\omega}$ a real random variable.
\end{enumerate}
Moreover, under the linearity condition necessarily $\lambda^{\ast}_{\omega}=\frac{1}{p-d}\mathbb{E}\left[\|Q_cZ\|^2|P_cZ\right]$.
\end{lemma}

\begin{proof}
Let us begin by the easiest way : (\ref{e2}) $\Rightarrow$ (\ref{e1}). Let $H$ be any orthonormal matrix as described in (\ref{e1}). Because $HQ_cH^T=I-HP_cH^T=Q_c$, by multiplying (\ref{e2}) on the left side by $H$ and on the right side by $H^T$, we find that
\begin{equation*}
\var(HZ|P_cZ)=\lambda^{\ast}_{\omega}Q_c=\var(Z|P_cZ).
\end{equation*}

The other way is based on a good choice of the matrix $H$. Let $\gamma$ be a unit vector of $E_c^{\perp}$, and define $H=I-2\gamma \gamma^T$. Clearly, $H$ is symmetric and satisfies to the requirement of (\ref{e1}). So that, we have the equation
\begin{equation*}
\var(Z|P_cZ)=(I-2\gamma \gamma^T)\var(Z|P_cZ)(I-2\gamma \gamma^T),
\end{equation*}   
developing the right hand side, it follows that
\begin{equation*}
\var(Z|P_cZ)\gamma \gamma^T= 2 \var(\gamma^T Z|P_cZ)\gamma\gamma^T-\gamma\gamma^T\var(Z|P_cZ),
\end{equation*}
and finally, multiplying by $\gamma$ on the right, we find
\begin{equation}\label{lambda}
\var(Z|P_cZ)\gamma=\var(\gamma^TZ|P_cZ)\gamma.
\end{equation}
Therefore, any $\gamma\in E_c^{\perp}$ is an eigenvector of $\var(Z|P_cZ)$ and thus, $E_c$ is an eigenspace of this matrix. Denote by $\lambda^{\ast}_{\omega}$ the eigenvalue associated to $E_c^{\perp}$. Since the columns of $Q_c$ are vectors of $E_c^{\perp}$, we have
\begin{equation*}
\var(Z|P_cZ)Q_c=\lambda^{\ast}_{\omega}Q_c,
\end{equation*} 
which implies that
\begin{equation*}
\var(Z|P_cZ)=\var(Q_cZ|P_cZ)=\lambda^{\ast}_{\omega}Q_c,
\end{equation*}
and (\ref{e1}) $\Rightarrow$ (\ref{e2}) is completed.

The value of $\lambda^{\ast}_{\omega}$ can be given by equation (\ref{lambda}). Clearly, under the linearity condition we have for every unit vector $\gamma\in E_c^{\perp}$, 
\begin{equation*}
\lambda^{\ast}_{\omega}=\var(\gamma^TZ|P_cZ)=\E[(\gamma^TZ)^2|P_cZ],
\end{equation*}
and hence it suffices to take $\gamma=\frac{1}{\sqrt{p-d}}\sum_{k=1}^{p-d} \gamma_k $ where $(\gamma_1,...,\gamma_{p-d})$ is an orthonormal basis of $E_c^{\perp}$, to obtain
\begin{equation*}
\lambda^{\ast}_{\omega}=\frac{1}{p-d}\mathbb{E}\left[\|Q_cZ\|^2|P_cZ\right].
\end{equation*}
\end{proof}

\begin{remark}\label{ccvdcv}
Here we compare CCV and DCV. Each existing method being based on close but sometimes different assumptions, it is difficult to build a complete sketch of the assumption sets used. Let us have a look to the interaction with the spherical assumption. First, Proposition \ref{bryc} informs us that coupling the CCV condition and the spherical assumption is equivalent to normality. But in our case, the sphericity implies DCV. Indeed, if $Z$ is spherical, then its distribution is invariant by any orthogonal transformation, and we have for any measurable function $f$ and for any orthogonal matrix $H$,
\begin{equation*}
\mathbb{E}[ZZ^T f(P_cZ)]=\mathbb{E}[HZZ^TH^Tf(P_cHZ)].
\end{equation*} 
In particular, the previous equation is true for any $H$ which leaves invariant vectors of $E_c$ and we obtain (\ref{e1}) of Lemma \ref{lemma} which is equivalent to DCV. Thus, we have just proved that the spherical assumption implies DCV.
\end{remark}

Theorem \ref{th2} is the analogue of Theorem \ref{th1} for TF2. 

\begin{theorem}\label{th2}
Define the matrix $M_{\psi}=\E[ZZ^T\psi(Y)]$. Assume that $Z$ satisfies Assumptions \ref{linearity} and \ref{variance} and has a finite second moment. Then, for every measurable function $\psi:\mathbb{R}\rightarrow \mathbb{R}$ such that $\E[ZZ^T\psi(Y)]<\infty$, we have
\begin{equation*}
\spann(M_{\psi}-\lambda_{\psi}^{\ast}I)\subset E_c,
\end{equation*}
with $\lambda_{\psi}^{\ast}=\frac{1}{p-d}\mathbb{E}\left[\|Q_cZ\|^2\psi(Y)\right]$.
\end{theorem}
\begin{proof}
To make a complete proof, we need to show that all the vectors in $E_c^{\perp}$ are eigenvectors of the symmetric matrix $M_{\psi}-\lambda_{\psi}^{\ast} I$ associated to the eigenvalue $0$. The existence of the CS ensures that
\begin{equation*}
M_{\psi}-\lambda_{\psi}^{\ast}I=\mathbb{E}[(\mathbb{E}[ZZ^T|P_cZ]-\lambda^{\ast}_{\omega}I)\psi(Y)],
\end{equation*}
besides, thanks to the linearity condition and DCV, we have
\begin{equation*}
\mathbb{E}[ZZ^T|P_cZ]= \lambda^{\ast}_{\omega}Q_c+P_cZZ^TP_c.
\end{equation*}
Thus, for any $\gamma\in E_c^{\perp}$ we have $(M_{\psi}-\lambda_{\psi}^{\ast}I) \gamma=0$ and the proof is completed.
\end{proof}

In practice, because $\lambda_{\psi}^{\ast}$ is unknown, it seems difficult to use Theorem \ref{th2}. Nevertheless, we do not really need to know this particular eigenvalue because a consequence of Theorem \ref{th2} is that $E_c^{\perp}$ is an eigenspace of the matrix $M_{\psi}$ associated to the eigenvalue equal to $\lambda_{\psi}^{\ast}$. Therefore, if the dimension of $E_c^{\perp}$ is large, then the spectrum of $M_{\psi}$ would have an accumulation of eigenvalues equal to $\lambda_{\psi}^{\ast}$. What we expect is that the other eigenvalues will be different from $\lambda_{\psi}^{\ast}$. If it is true, all the directions of $E_c$ could be recovered and this eigenvalue problem is the topic of the next section.

\section{Covering the central subspace}\label{s4}
In this section, we find that a way to obtain an exhaustive characterization of the CS for TF1 and TF2 is to consider many $\psi$ function. As usual, we begin with TF1 and conclude by TF2.

\subsection{\textbf{Order $1$ test function}}
As a consequence of Theorem \ref{th1}, spaces generated by $(\E[Z\psi_{1}],...,\E[Z\psi_{k}])$ are included in $E_c$. Our goal is to obtain the converse inclusion. Because TF1 is an extending of SIR, this one has a central place in the following argumentation. We start by giving a necessary and sufficient condition for covering the entire CS with SIR. Then under the same condition we extend SIR to a new class of methods. 
   
\begin{assumption}\label{ft1}
For every nonzero vector $\eta\in E_c$, $\E[\eta^TZ|Y]$ has a nonzero variance.
\end{assumption}

Equation (\ref{patho}) provides a regression model for which a direction of $E_c$ is almost surely orthogonal to $\E[Z|Y]$. It is clear that this kind of situation is no longer allowed by the previous assumption. However, TF2 is designed to handle such pathological cases.

\begin{lemma}\label{coversir}
If $Z$ satisfies Assumption \ref{linearity} and has a finite second moment, then Assumption \ref{ft1} implies that $\spann(M_{SIR})=E_c$ and conversely.
\end{lemma}
\begin{proof}
Under the linearity condition, $\spann(M_{SIR})=E_c$ is equivalent to $\eta^T M_{SIR}\ \eta >0$ for every $\eta\in E_c$, which is another formulation of Assumption (\ref{ft1}). 
\end{proof}

We now extend Lemma \ref{coversir} to TF1, the aim is to provide the same results replacing the conditional expectation $\E[Z|Y]$ in $M_{SIR}$ by some known family of functions. To state the following theorem, we introduce the function space $L_1\left(\theta(y)\mu(\dd y)\right)$ defined as 
\begin{equation*}
L_1\left(\theta(y)\mu(\dd y)\right)=\{ u:\mathbb{R}\rightarrow \mathbb{R} ; \int_{\mathbb{R}} |u(y)|\theta(y)\mu(\dd y)<+\infty \},
\end{equation*} 
where $\theta:\mathbb{R}\rightarrow \R_+$ and $\mu$ a real measure.

\begin{theorem}\label{anypsi}
Assume that $Z$ and $Y$ satisfy Assumptions \ref{linearity} and \ref{ft1}. Assume also that $Z$ has a finite second moment. If $\Psi$ is a total countable family in the space $L_1(\mathbb{E}[\|Z\||Y=y] P_Y(dy))$, then we can extract a finite subset $\Psi_H$ of $\Psi$ such that 
\begin{equation*}
\spann \left(\E[Z\psi(Y)],\ \psi\in \Psi_H\right)=E_c.
\end{equation*}
\end{theorem}

\begin{proof}
Lemma \ref{coversir} provides that $\{\E[Z\E[Z_k|Y]],\ k=1,...,p\}$ is a generator of $E_c$. First, let us  show that any vector of this family can be approximated by $\E[Z\phi(Y)]$, where $\phi$ is a linear combination of functions in $\Psi$. Let $\varepsilon>0$ and $k\in \{1,...,p\}$. Since $\Psi$ is a total family in $L_1(\mathbb{E}[\|Z\||Y=y] P_Y(dy))$, there exists $\phi_k$ a finite linear combination of functions in $\Psi$ such that
\begin{equation*}
\mathbb{E}\left[\mathbb{E}[\|Z\||Y]\ |\phi_k(Y)-\E[Z_k|Y]|\right]\leq \varepsilon,
\end{equation*}
besides, we have
\begin{eqnarray*}
\|\mathbb{E}[Z\phi_k(Y)]-\mathbb{E}[Z\E[Z_k|Y]]\| &= & \left\|\ \mathbb{E}\left[\E[Z|Y]\  (\phi_k(Y)-\E[Z_k|Y])\right]\right\| \\
&\leq & \mathbb{E}\left[\E[\|Z\|\ |Y] \ \left|\phi_k(Y)-\E[Z_k|Y]\right|\ \right],
\end{eqnarray*}
and therefore,
\begin{equation}\label{close}
 \|\mathbb{E}[Z\phi_k(Y)]-\mathbb{E}[Z\E[Z_k|Y]]\|\leq \varepsilon.
\end{equation}
Here an important point is that $\E[Z\phi_k(Y)]\in E_c$, it implies that
\begin{equation}\label{subset}
\spann\left(\E[Z\phi_k(Y)],\ k=1,...,p\right)\subset \spann(M_{SIR}),
\end{equation} 
Moreover, (\ref{close}) and the continuity of the determinant involve that the rank of the set of vectors $\E[Z\phi_k(Y)]$ is equal to $d$ if $\varepsilon$ is small enough. Then, instead of an inclusion (\ref{subset}) become an equality and we complete the proof by recalling that each $\phi_k$ is a linear combination of a finite number of functions in $\Psi$. 
\end{proof}

Theorem \ref{anypsi} assumes that the family is total. Some mild conditions can be found in~\citet*{coudene2002}. Let us recall their main result.

\begin{theth}\label{thethm}{(Y. Coud\`ene)}
Let $p\in[0,\infty[$, $\mu$ a borelian probability measure on $[0,1]$, and $f_n:[0,1]\rightarrow \R$ a family of bounded measurable functions that separates the points :
\begin{equation*}
\forall x,y\in [0,1],\ x\neq y,\ \exists n\in \N \quad\text{such that}\quad f_n(x)\neq f_n(y).
\end{equation*} 
Then the algebra spanned by the functions $f_n$ and the constants is dense in $L_p([0,1],\mu)$.
\end{theth}

\begin{remark}\label{coudene}
Accordingly, we can apply Theorem \ref{anypsi} with any family of functions that separates the points, for example polynomials, complex exponentials or indicator functions. To make possible a simple use of this theorem we need to recall this result. If $u=(u_1,...,u_H)$ is a $\R^p$ vector family, then $\spann(uu^T)=\spann(u)$. Thus, if we denote by $\psi_1,...,\psi_H$ some elements of a family that separates the points, then the CS can be obtained by making an eigendecomposition of the order $1$ test function matrix associated to the functions $\psi_1,...,\psi_H$ defined as
\begin{equation*}
M_{TF1}=\sum_{h=1}^H \E[Z\psi_{h}(Y)] \E[Z\psi_h(Y)]^T. 
\end{equation*}
Especially, the eigenvectors associated to a nonzero eigenvalue of any order $1$ test function matrix span the CS. Moreover, as pointed out in \citet*{cook2005}, for $H$ large enough $\spann(\widetilde{M}_{SIR})=\spann(M_{SIR})$. A proof of this result can be obtained by Theorem \ref{anypsi}. By applying it with the indicator family of functions, it gives that
\begin{equation*}
\spann\left(\E[Z\mathds{1}_{\{Y\in I(h)\}}],\ h=1,...,H\right)=\spann(\widetilde{M}_{SIR})=\spann(M_{SIR})=E_c,
\end{equation*} 
if $H$ is sufficiently large. Also, SIR can be understood as a particular TF1. Expression (\ref{uut}) implies that 
\begin{equation*}
\widetilde{M}_{SIR}=\sum_{h=1}^H \frac{1}{p_h} \E[Z\mathds{1}_{\{Y\in I(h)\}}] \E[Z\mathds{1}_{\{Y\in I(h)\}}]^T,
\end{equation*}
hence, SIR is equivalent to TF1 realized with the weighted family of indicator functions $\left(\frac{\mathds{1}_{\{Y\in I(h)\}}}{\sqrt{p_h}}\right)$. More generally for any family of functions, the space spanned by $M_{TF1}$ is not change by a weighting with positive weight. Nevertheless it is no longer the case for the estimated space, and intuitively it seems that such a weighting could influence the convergence rate. The choice of the weights for the family of indicators is debated is section \ref{s51} thanks to a variance minimization. 
\end{remark}

\subsection{\textbf{Order $2$ test function}}
As described before, an important tool in this section is the eigendecomposition of the matrix $M_{\psi}$, therefore we try to be more comprehensive in introducing the following notation. Let $\lambda_{\psi}$ and $\lambda_{Y}$ be the functions $\mathbb{R}^p\rightarrow \mathbb{R}$ respectively defined by 
\begin{equation*}
\lambda_{\psi}(\eta)=\mathbb{E}[(\eta^T Z)^2\psi(Y)]\quad \text{and}\quad \lambda_{Y}(\eta)=\mathbb{E}[(\eta^T Z)^2|Y],
\end{equation*}
and notice that if $\eta$ is a unit eigenvector of $M_{\psi}$ (resp. $ \mathbb{E}[ZZ^T|Y]$), then $\lambda_{\psi}(\eta)$ (resp. $\lambda_Y(\eta)$) is equal to the eigenvalue of the matrix $M_{\psi}$ (resp. $ \mathbb{E}[ZZ^T|Y]$) associated to $\eta$. However, recalling that $E_c^{\perp}$ is an eigenspace of $M_{\psi}$ and $\mathbb{E}[ZZ^T|Y]$, the functions $\lambda_{\psi}$ and $\lambda_Y$ are both constant on the centered spheres of $E_c^{\perp}$. Their respective values on the unit sphere of $E_c^{\perp}$ are noted $\lambda_{\psi}^{\ast}$ and $\lambda_Y^{\ast}$.  

\begin{definition}
Let $\psi$ be a measurable function. We call $\psi$-space and note $E_{\psi}$ the space
\begin{equation*}
E_{\psi}=\spann(M_{\psi}-\lambda_{\psi}^{\ast})=\spann\left(\eta\in B(0,1)\subset\mathbb{R}^p,\ M_{\psi}\eta = \lambda_{\psi}^{\ast} \eta \right)^{\perp}.
\end{equation*}
\end{definition}
Thanks to Theorem \ref{th2} we have already proved that under Assumption \ref{linearity} and \ref{variance} any $\psi$-space is included in $E_c$. However, nothing guarantees the existence of a $\psi$-space equal to $E_c$. We follow the same idea than for the order $1$ method, i.e. we consider some transformations of $Y$ belonging to a dense family. Nevertheless, the results are a little different because we provide the existence of a $\psi$-space equal to $E_c$. A unique additional assumption is needed.
\begin{assumption}\label{cond2}
\begin{equation*}
\forall \eta \in E_c,\ \|\eta\|=1\ \ \ \ \mathbb{P}\left(\mathbb{E}\left[(\eta^TZ)^2|Y\right]=\mathbb{E}\left[\left. \frac{\|Q_cZ\|^2}{p-d}\right| Y\right]\right)<1.
\end{equation*}
\end{assumption}

\begin{remark}
Assumption \ref{cond2} takes the same approach as \citet*{li2007}. As it is highlighted in Remark \ref{ccvdcv}, our set of assumptions is weaker than their beacause DCV has replaced CCV. To match their context, assume that CCV condition is satisfied. Then clearly, Assumption \ref{cond2} becomes ``$\E[(\eta^T Z)^2|Y]$ is nondegenerate", i.e. is not a.s. a constant. Otherwise, TF1 allows an exhaustive estimation of the CS provided that $\E[(\eta^TZ)|Y]$ is nondegenerate. Thus the exhaustiveness condition of TF is the union of the two previous and it gives
\begin{equation*}
\E[(\eta^T Z)^2|Y]\quad\text{or}\quad  \E[(\eta^TZ)|Y]\quad \text{is nondegenerate,} 
\end{equation*}
which is the same than the one provided for DR in \citet*{li2007}. Accordingly, TF evolved in a more general context given by DCV but the assumptions ensuring its exhaustiveness are as weak as the one in the literature.
\end{remark}

In the proof of the following theorem we will need Lemma \ref{lemmematrix} and Proposition \ref{propmatrix} which are stated and demonstrated in the appendix.

\begin{theorem}\label{apsi}
Assume that $Z$ and $Y$ satisfy Assumptions \ref{linearity}, \ref{variance} and \ref{cond2}. Assume also that $Z$ has a finite second moment, then if $\Psi$ is a total countable family in the space $L_1(\E[\|Z\|^2\ |Y=y] P_Y(dy))$, there exists $\psi$ a finite linear combination of functions in $\Psi$ such that
\begin{equation*}
E_{\psi}=E_c.
\end{equation*}
\end{theorem}
\begin{proof}
Let $\Psi$ be a total countable family in $L_1(\E[\|Z\|^2\ |Y=y] P_Y(dy))$. By Theorem \ref{th2}, $E_c^{\perp}\subset E_{\psi}^{\perp}$ for any $\psi$. Then it suffices to show that there exists $\psi$ a finite linear combination of functions in $\Psi$ such that $\dim(E_{\psi})=\rank(M_{\psi}-\lambda_{\psi}^{\ast}I)=d$. In the basis $(P_1,P_2)$, where $P_1$ and $P_2$ are respectively basis of $E_c$ and $E_c^{\perp}$, the matrix $M_{\psi}-\lambda_{\psi}^{\ast}I$ can be written as
\begin{equation*}
\matdeux{N_{\psi}}{0}{0}{0},
\end{equation*}  
with $N_{\psi}=P_1^T(M_{\psi}-\lambda_{\psi}^{\ast})P_1$. Notice that the space
\begin{equation*}
\mathcal{M}=\{N_{\psi},\ \psi=\sum_h \alpha_h \psi_h\},
\end{equation*}
is a linear subspace of the symmetric matrices with dimension $d\times d$. In the basis $(P_1,P_2)$, Assumption (\ref{cond2}) becomes
\begin{equation*}
\forall \eta\in \R^d,\quad \P(\eta^T N_{Y}\eta=0)< 1,
\end{equation*}
with $N_Y=P_1^T(M_{Y}-\lambda_{Y}^{\ast})P_1$. Clearly, this implies that
\begin{equation}\label{newcond}
\forall \eta \in \R^d ,\quad \exists \psi, \quad \eta^T N_{\psi} \eta\neq0,
\end{equation}
and because $\Psi$ is a total family in $L_1(\E[\|Z\|^2\ |Y=y] P_Y(dy))$, the function $\psi$ in the previous equation could be a finite linear combination of functions in $\Psi$ and then $N_{\psi}\in \mathcal{M}$. Thus the proof consists in showing that given a linear subspace $\mathcal{M}\subset \R^{d\times d}$ of symmetric matrices, if (\ref{newcond}) is checked, then there exists an invertible matrix in $\mathcal{M}$. The contrapositive is the statement of Proposition \ref{propmatrix}.
\end{proof}

Theorem \ref{apsi} states the existence of a $\psi$-space equal to $E_c$, yet it does not provide an explicit form of such a $\psi$. Hence, we set out the following corollary.

\begin{corollary}\label{sumpsi}
Assume that $Z$ and $Y$ satisfy Assumptions \ref{linearity}, \ref{variance} and \ref{cond2}. Assume also that $Z$ has a finite second moment then, if $\Psi$ is a total countable family in the space $L_1(\E[\|Z\|^2\ |Y=y] P_Y(dy))$, we have
\begin{equation*}
 \underset{\Psi_H}{\oplus} E_{\psi} =E_c,
\end{equation*}   
where $\Psi_H$ is a finite subset of $\Psi$.
\end{corollary}

\begin{proof}
From Theorem \ref{apsi} we have $E_{\psi} =E_c$ where $\psi =\sum_{h=1}^H\alpha_h \psi_h$. Hence, we need to show that $E_{\psi}\subset \oplus E_{\psi_h} $ since the other inclusion is trivial. Suppose that it exists $\eta\in E_{\psi}$ with norm $1$ such that $\eta\perp \oplus E_{\psi_h} $. Then by definition, for every $h=1,...,H$,
\begin{equation*}
M_{\psi_h}\eta = \lambda_ {\psi_h}^ \ast \eta,
\end{equation*}
and we obtain
\begin{equation*}
M_{\psi}\eta = \sum_{h=1}^H \alpha_h \lambda_ {\psi_h}^ \ast \eta=\lambda_{\psi}^ \ast \eta.
\end{equation*}
which is impossible because $\eta\in E_{\psi}$. 
\end{proof}

Corollary \ref{sumpsi} is the counterpart of Theorem \ref{anypsi} for TF2. Nevertheless, it seems difficult to use it in practice because it requires an eigendecomposition of a large number of matrices. Besides, Theorem \ref{apsi} is the cornerstone of TF2. Using the theorem quoted in Remark \ref{coudene}, we provide order $2$ methods based on families of functions that separate the points. For each such family, it exists a function $\psi$ such that the associated $\psi$-space is equal to the CS.

\section{Choice of the test function. Asymptotic variance minimization}\label{s5}
This section is divided into two paragraphs. First, we study the case of the family of indicator functions for TF1 and secondly, we are interested in finding the best $\psi$ for TF2. Clearly, for the order $1$ method we need at least $d$ functions to recover the CS whereas for the order $2$, as we showed before, we can expect to find a function $\psi$ such that $M_{\psi}$ covers all the directions of the CS. This is the reason why we fix the class of function in the first paragraph and we search a unique function in the second paragraph. 

\subsection{Order 1 test function: optimality among the indicators}\label{s51}
In this section, we develop a test function plug-in method based on the minimization of the variance estimation in the case of the family of indicator functions for $\Psi_H$. Theorem \ref{anypsi} and Remark \ref{coudene} imply that the whole subspace $E_c$ can be covered by the family of vectors $\{\E[Z\mathds{1}_{\{Y\in I(h)\}}],h=1,...,H\}$ for a suitable partition $I(h)$. Actually, it is possible to extract $d$ orthogonal vectors living in the space spanned by this family, and then it provides us a basis of the CS. This procedure is realized by SIR. Nevertheless, the issue here is somewhat more complicated, we want to find $d$ orthogonal vectors that have the minimal asymptotic mean squared error for the estimation of the projection $P_c$. We define
\begin{equation}\label{mse1}
\text{MSE}=\E\left[\|P_c-\^P_n\|^2\right],
\end{equation}
where $\|\cdot \|$ stands for the Frobenius norm and $\^P_n$ is derived from the family of vector $\^\eta=(\^\eta_1,...,\^\eta_d)$ defined as
\begin{equation*}
\^\eta_k=\frac{1}{n}\sum_{i=1}^n Z_i\psi_k(Y_i)\quad\text{with}\quad \psi_k(Y)=(\mathds{1}_{\{Y\in I(1)\}},...,\mathds{1}_{\{Y\in I(H)\}})\alpha_k =\mathds{1}_Y^T \alpha_k,
\end{equation*} 
where $\alpha_k \in \R^{H}$. Besides, we introduce $\eta=(\eta_1,...,\eta_d)$ with $\eta_k=\E[Z\psi_k(Y)]$. Consequently, we aim at minimizing MSE according to the family $(\psi_k)_{1\leq k\leq d}$, or equivalently according to the matrix $\alpha=(\alpha_1,...,\alpha_d)\in \R^{H\times d}$. Moreover, since we have
\begin{eqnarray}\label{mse}
\text{MSE}&=& \E[\tr(P-\^P_n)^2]\nonumber\\
&=& d +\E[\^d-2\tr((I-Q_c)\^P_n)]\nonumber\\
&=& \E[d-\^d]+2\E[\tr(Q_c\^P_n)],
\end{eqnarray}
and we suppose that $d$ is known, the minimization of MSE results only on the minimization of the second term in the previous equality. Hence, this naturally leads us to the minimization problem 
\begin{equation*}
\underset{\alpha}{\text{min}}\ \lim_{n\rightarrow \infty} n\E[\tr(Q_c\^P_n)],
\end{equation*}
under the constraint of orthogonality of the family $(\eta_k)_{1\leq k\leq d}$.
For a more comprehensive approach, we choose to minimize the expectation of the limit in distribution, instead of the limit of the expectation when $n$ goes to infinity, of the sequence $n\tr(Q_c\^P_n)$. To set out clearly the next proposition, let us introduce some notations. Define the matrices 
\begin{eqnarray*}
C = (C_1,...,C_H)&\quad\text{with} & C_h=\E[Z \mathds{1}_{\{Y\in I(h)\}}], \\ 
D = \diag d_h & \quad\text{with} & d_h= \left(\E[\|Q_cZ\|^2\mathds{1}_{\{Y\in I(h)\}}]\right),
\end{eqnarray*}
and 
\begin{equation*}
G=D^{-\frac{1}{2}}C^T CD^{-\frac{1}{2}}.
\end{equation*}
The matrix $G$ is the Gram matrix of the vector family $(C_h/\sqrt{d_h})_{1\leq h\leq H}$, Theorem \ref{anypsi} and Remark \ref{coudene} ensure that its rank is equal to $d$. Besides, $G$ is diagonalisable and so we define $P=(P_1P_2)\in \R^{p\times (d+(p-d))}$ such that  
\begin{equation*}
P^T G P = \begin{pmatrix} D_0 &  0 \\ 0&0\end{pmatrix}, 
\end{equation*}
where $D_0\in \R^{d\times d}$.

\begin{proposition}\label{alpha}
The random variable $n \tr(Q_cP_n)$ has a limit in law $W_ \alpha  $ as $n\rightarrow \infty$. The minimization problem
\begin{equation}\label{minpr}
\underset{\alpha}{\min}\  \E\left[W_{\alpha}\right] \quad\text{u.c.}\quad \eta^T\eta=Id,
\end{equation}
has a unique solution, up to orthogonal transformations, given by
\begin{equation*}
\alpha=D^{-\frac{1}{2}} P_1 D_0^{-\frac{1}{2}}.
\end{equation*}

\end{proposition}

\begin{proof}
We first calculate the expectation of the limit in law of the sequence $n\tr(Q_c\^P_n)$ and then we solve the optimization problem. Since 
\begin{eqnarray*}
n\tr(Q_c\^P_n) &=& n\tr(\^\eta^TQ_c\^\eta\ (\^\eta^T \^\eta)^{-1})\\
&=& \tr(\sqrt{n}(\^\eta^T-\eta^T)Q_c\sqrt{n}(\^\eta-\eta) (\^\eta^T \^\eta)^{-1}),
\end{eqnarray*}
Slutsky's theorem and the continuity of the operator $\tr(\cdot)$ provides that $n\tr(Q_c\^P_n)$ converges to $\tr(\delta^TQ_c\delta)$ in distribution, where $\delta\in \R^{p\times d}$ is the limit in law of the sequence $\sqrt{n}(\^\eta-\eta)$, i.e. a normal vector with mean $0$. Thus it remains to calculate the expectation of this limit, notice that
\begin{eqnarray*}
\E\left[W_{\alpha}\right]=\E\left[\tr(\delta^TQ_c\delta)\right] &=& \sum_{k=1}^d \tr\left(Q_c\E[\delta_{k}\delta_{k}^T]\right),
\end{eqnarray*}
where $\delta_{k}$ stands for the limit in law of the sequence $\sqrt{n}(\^\eta_k-\eta_k)$. Finally, since its variance is equal to $\var(Z\psi_k(Y))$ and using the linearity condition, we have
\begin{equation}\label{lim}
\E\left[W_{\alpha}\right]=\sum_{k=1}^d \E\left[\|Q_cZ\|^2 \psi_k(Y)^2\right].
\end{equation}

Now let us reformulate the minimization problem in terms of matrix $\alpha$. First, from (\ref{lim}) and using that the $I(h)$ are pairwise disjoint, we have 
\begin{equation}\label{min}
\E\left[W_{\alpha}\right] = \sum_{k=1}^d \alpha_k^T \E[\|Q_c Z\|^2 \mathds{1}_Y\mathds{1}_Y^T]\alpha_k =\tr(\alpha^T D \alpha ),
\end{equation} 
and also, 
\begin{equation}\label{cons}
\eta^T\eta= \alpha^T C^T C \alpha = (D^{\frac{1}{2}}\alpha)^TG D^{\frac{1}{2}}\alpha.
\end{equation}
From (\ref{min}) and (\ref{cons}) we set out the equivalent minimization problem 
\begin{equation*}
\underset{\alpha}{\text{min}}\  \tr\left(\alpha^T D \alpha\right)\quad\text{u.c.}\quad (D^{\frac{1}{2}}\alpha)^TG D^{\frac{1}{2}}\alpha =Id,
\end{equation*}
then, from the variable change $U=P^T D^{\frac{1}{2}} \alpha$ we derive
\begin{equation*}\label{minpr2}
\underset{U}{\text{min}}\ \tr( U^TU)\quad\text{u.c.}\quad U^T \matdeux{D_0}{0}{0}{0} U =Id.
\end{equation*}
By writing $U^T=(U_1^T,U_2^T)$ we notice that there is no constraint on $U_2$, which implies that $U_2=0$. Consequently, it remains to solve 
\begin{equation}\label{minpr3}
\underset{U_1}{\text{min}}\ \tr( U_1U_1^T)\quad\text{u.c.}\quad U_1 U_1^T =D_0,
\end{equation}
where $U_1\in \R^{d\times d}$. Clearly, in (\ref{minpr3}) the quantity to minimize is fixed by the constraint. Then, a solution of it is given by $U_1=D_0^{-\frac{1}{2}} H$ where $H$ is any orthogonal matrix. Hence, the solution of (\ref{minpr}) is
\begin{equation*}
\alpha=D^{-\frac{1}{2}} P U=D^{-\frac{1}{2}} P_1 D_0^{-\frac{1}{2}} H,
\end{equation*}
where $H$ is any orthogonal matrix.
\end{proof}

To make a link with other methods and facilitate the programming of TF1, let us explain the solution in another way. Instead of explaining the solution in terms of weight we put on the indicator functions, we explain it in terms of vectors $\eta_k$ associated to these weights. First notice that, with the chosen notation
\begin{equation*}
D^{-\frac{1}{2}}C^TC D^{-\frac{1}{2}} P_1=P_1 D_0,
\end{equation*} 
multiplying by $CD^{-\frac{1}{2}}$ on the left and by $D_0^{-\frac{1}{2}}$ on the right, it gives
\begin{equation*}
C D^{-1}C^T C D^{-\frac{1}{2}} P_1 D_0^{-\frac{1}{2}}=C D^{-\frac{1}{2}} P_1 D_0^{-\frac{1}{2}} D_0.
\end{equation*}
Defining an order $1$ test function matrix $\widetilde{M}_{TF1}=C D^{-1}C^T$, and noting that $\eta=C D^{-\frac{1}{2}} P_1 D_0^{-\frac{1}{2}}$, the previous equation is equivalent to
\begin{equation*}
\widetilde{M}_{TF1}\eta=\eta D_0.
\end{equation*}
Thus, since $\widetilde{M}_{TF1}$ has the same rank as $G$, we have showed that the vectors $\eta_k$ derived from the optimal weight family, are the eigenvectors of $\widetilde{M}_{TF1}$ associated to nonzero eigenvalues. Besides, it is easy to verify that the previous development is still true when each quantity is replaced by its estimate. Therefore in practice, we have to make the eigendecomposition of an estimator of the matrix $\widetilde{M}_{TF1}$.

As it is stated in the introduction of section \ref{basic}, the SIR estimator is obtained thanks to an eigendecomposition of the matrix $\widetilde{M}_{SIR}$, while our matrix of interest here is $\widetilde{M}_{TF1}$. To compare both methods, we write there expressions as follows
\begin{equation}\label{poids}
\widetilde{M}_{SIR}=\sum_{h=1}^H \frac{C_h  C_h^T}{p_h},\quad \quad\widetilde{M}_{TF1}=\sum_{h=1}^H \frac{C_h  C_h^T}{d_h}.
\end{equation}
As we noticed before, SIR is really closed to the order one test function method proposed here, both methods try to obtain the information contains in the slices through the $C_h$. This information is collected more rapidly thanks to TF1 because it minimizes the criterion (\ref{mse1}), and as a consequence the convergence rate would be better. This idea is supported by the expression of $\widetilde{M}_{TF1}$ in which bad slices are less weighted. When $H\rightarrow\infty$, $\widetilde{M}_{SIR}\rightarrow M_{SIR}$ and clearly $\widetilde{M}_{TF1}$ converge to
\begin{equation*}
M_{TF1}=\E\left[Z\frac{ \E[Z|Y]}{\E[\|Q_cZ\|^2|Y]}\right].
\end{equation*}
As a consequence of (\ref{poids}), the TF1 variance minimization with indicators requires the knowledge of $Q_c$. Therefore we set out a plug-in method to estimate $Q_c$.

\textbf{TF1 Algorithm:}
\begin{enumerate}
\setcounter{enumi}{-1}
\item Standardization of $X$ into $Z$. Initialize $\^Q_c=I$.
\item Compute
\begin{equation*}
\^d_{h}=\frac 1 n \sum_{i=1}^n \|\^Q_c Z_i\|^2\mathds{1}_{\{Y_i\in 
I(h)\}} , \quad\^C_h=\frac 1 n\sum_{i=1}^n Z_i\mathds{1}_{\{Y_i\in 
I(h)\}} 
\end{equation*}
\begin{equation*}
\text{and}\quad \^M=\sum_{h=1}^H \frac{\^C_h\^C_h^T}{\^d_h}.
\end{equation*}

\item Extract $\^\eta=(\^\eta_1,...,\^\eta_d)$: the $d$ eigenvectors of $\^M$ with largest eigenvalues.
\item $\^Q_c=I-\^\eta\^\eta^T$.
\end{enumerate}
Steps $1$ to $3$ are repeated until convergence is achieved and then $\^\eta$ is the estimated basis of the standardized CS derived from TF1. The estimated directions of the CS are $\^\Sigma^{-\frac 1 2}\^\eta$. At the end of the paper, this method is tested and compared to SIR using simulations.

\subsection{Order 2 test function: Optimality among the measurable functions}\label{s52}

Here we have a different approach than for TF1, we aim at finding the optimal $\psi$ such that the variance error is minimal. Recall that $M_{\psi}=\E[ZZ^T\psi(Y)]$, we have already proved that the eigenvectors of this matrix can be decomposed into two blocks : the one associated to the eigenvalue $\lambda_{\psi}^{\ast}$ and the other which necessarily belongs to $E_c$. Therefore, $\^P_n$ is derived from the eigenvectors associated to the eigenvalues different from $\lambda_{\psi}^{\ast}$, and so we decided to express $\^P_n$ in the following way. Theorem \ref{apsi} guarantees the existence of $\psi$ such that $E_{\psi}=E_c$. Based on this result, suppose that we are able to differentiate each eigenvalue associated to an eigenvector in $E_c$ from $\lambda_{\psi}^{\ast}$. Then we can find a contour $\mathcal{C}$ which encloses the eigenvalues different from $\lambda_{\psi}^{\ast}$, and finally we can write $P_c$ and its estimator $\^P_n$ as
\begin{equation*}
P_c= \oint_{\mathcal{C}} (Iz-M_{\psi})^{-1} \dd z \quad\text{and}\quad \^P_n= \oint_{\mathcal{C}} (Iz-\^M_{\psi})^{-1} \dd z,
\end{equation*}
where $\^M_{\psi}=\frac{1}{n}\sum_{i=1}^n Z_iZ_i^T\psi(Y_i)$. As we did for TF1, we aim at minimizing the $MSE$ through the quantity $\E[\tr(Q_c\^P_n)]$ (see equation (\ref{mse})). We first calculate the limit in law of the random variable $n \tr(Q_c\^P_n)$, as $n$ goes to infinity and then we derive its expectation. The next proposition is dedicated to this calculus. 

\begin{proposition}\label{calvar}
Let $W_{\psi}$ be the limit in law of the random variable $n \tr(Q_c\^P_n)$, then 
\begin{equation*}
\E[W_{\psi}]=\tr\left( \E\left[ZZ^T \|Q_cZ\|^2 \psi(Y)^2\right] P_c(P_cM_{\psi}-I\lambda_{\psi}^{\ast})^{-2} \right).
\end{equation*}
\end{proposition}

\begin{proof}
We have 
\begin{eqnarray*}
Q_c\^P_n &=& Q_c  (\^P_n-P_c)\\
&=& Q_c \oint_{\mathcal{C}} (Iz-\^M_{\psi})^{-1} - (Iz-M_{\psi})^{-1}\dd z\\
&=& Q_c \oint_{\mathcal{C}} (Iz-\^M_{\psi})^{-1}(M_{\psi}-\^M_{\psi}) (Iz-M_{\psi})^{-1}\dd z,
\end{eqnarray*} 
and then, we derive
\begin{multline}\label{multi}
Q_c\^P_n = Q_c \oint_{\mathcal{C}} (Iz-M_{\psi})^{-1}(M_{\psi}-\^M_{\psi})(Iz-M_{\psi})^{-1}\dd z \\ + Q_c \oint_{\mathcal{C}}(Iz-\^M_{\psi})^{-1} (M_{\psi}-\^M_{\psi})(Iz-M_{\psi})^{-1} (M_{\psi}-\^M_{\psi})  (Iz-M_{\psi})^{-1}\dd z.
\end{multline} 
Consider the trace of the first term of equation (\ref{multi}), since $Q_c$ and $(Iz-M_{\psi})^{-1}$ commute we have
\begin{multline*}
\tr\left(Q_c \oint_{\mathcal{C}} 
(Iz-M_{\psi})^{-1}(M_{\psi}-\^M_{\psi})(Iz-M_{\psi})^{-1}\dd 
z\right)=\\ \tr\left((M_{\psi}-\^M_{\psi})\oint_{\mathcal{C}} Q_c  (Iz-M_{\psi})^{-2}\dd z \right).
\end{multline*}
Besides, it is clear that 
\begin{equation}\label{diag}
Q_c(Iz-M_{\psi})^{-1}=\frac{Q_c}{(z-\lambda_{\psi}^{\ast})},
\end{equation}
and recalling that $\lambda_{\psi}^{\ast}$ is outside $\mathcal{C}$, we 
have $\oint_{\mathcal{C}} \frac{1}{(z-\lambda_{\psi}^{\ast})^{-2}} \dd 
z=0$ and then (\ref{multi}) implies that
\begin{multline*}
 \tr\left( Q_c\^P_n\right) = \tr\bigg(Q_c 
 \oint_{\mathcal{C}}(Iz-\^M_{\psi})^{-1} 
 (M_{\psi}-\^M_{\psi})(Iz-M_{\psi})^{-1}\\
  (M_{\psi}-\^M_{\psi})  (Iz-M_{\psi})^{-1}\dd z \bigg).
\end{multline*}
Denote by $\Delta$ the limit in law of $\sqrt{n}(M_{\psi}-\^M_{\psi})$, 
since $\^M$ goes to $M$ in probability, Slutsky's Theorem implies the 
convergence $n\tr\left( Q_c\^P_n\right) \overset{d}{\longrightarrow}  
W_{\psi}  $ with 
\begin{equation*}
W_{\psi}=\tr\left( Q_c 
\oint_{\mathcal{C}} (Iz-M_{\psi})^{-1}\Delta (Iz-M_{\psi})^{-1}\Delta  
(Iz-M_{\psi})^{-1}\dd z\right)
\end{equation*}
Here we use equation (\ref{diag}) to derive 
\begin{equation}\label{cal1}
W_{\psi} =\tr\left(Q_c\Delta\oint_{\mathcal{C}}\frac{(Iz-M_{\psi})^{-1}}{(z-\lambda_{\psi}^{\ast})^2}\dd z\Delta Q_c\right),
\end{equation}
and the integral inside (\ref{cal1}) can be calculated the following way. Splitting it into two terms and using (\ref{diag}), we obtain
\begin{eqnarray*}
\oint_{\mathcal{C}}\frac{(Iz-M_{\psi})^{-1}}{(z-\lambda_{\psi}^{\ast})^2} \dd z &=& \oint_{\mathcal{C}}\frac{P_c(Iz-M_{\psi})^{-1}}{(z-\lambda_{\psi}^{\ast})^2}\dd z +\oint_{\mathcal{C}}\frac{Q_c(Iz-M_{\psi})^{-1}}{(z-\lambda_{\psi}^{\ast})^2}\dd z\\
&=& \oint_{\mathcal{C}}\frac{P_c(Iz-M_{\psi})^{-1}}{(z-\lambda_{\psi}^{\ast})^2}\dd z +Q_c\oint_{\mathcal{C}}\frac{1}{(z-\lambda_{\psi}^{\ast})^3}\dd z,
\end{eqnarray*}
the last term in the previous equation is clearly equal to $0$. Concerning the first term, since for all $k=1,...,d$, we have
\begin{eqnarray*}
P_c\oint_{\mathcal{C}}\frac{(Iz-M_{\psi})^{-1}}{(z-\lambda_{\psi}^{\ast})^2} \dd z \eta_k &=& \eta_k \oint_{\mathcal{C}}\frac{(z-\lambda_{\psi}(\eta_k))^{-1}}{(z-\lambda_{\psi}^{\ast})^2} \dd z\\
&=& \frac{\eta_k}{(\lambda_{\psi}(\eta_k)-\lambda_{\psi}^{\ast})^{2}}\\
&=& P_c(P_cM_{\psi}-I\lambda_{\psi}^{\ast})^{-2} \eta_k,
\end{eqnarray*}
and since all the vectors in $E_c^{\perp} $ belong to the kernel of this matrix, we get
\begin{equation*}
P_c\oint_{\mathcal{C}}\frac{Iz-M_{\psi})^{-1}}{(z-\lambda_{\psi}^{\ast})^2} \dd z = P_c(P_cM_{\psi}-I\lambda_{\psi}^{\ast})^{-2}.
\end{equation*}
Injecting it in (\ref{cal1}), this leads us to  
\begin{equation*}
W_{\psi} = \tr\left( \Delta Q_c \Delta P_c(P_c M_{\psi}-I\lambda_{\psi}^{\ast})^{-2}\right),
\end{equation*}
and it remains to calculate its expectation. The linearity condition implies that $Q_cM_{\psi}P_c=0$, and we have
\begin{eqnarray*}
\E[\Delta Q_c \Delta P_c] &=& - \lim_{n\rightarrow \infty} n\E\left[(M_{\psi}-\^{M}_{\psi})Q_c\^{M}_{\psi}P_c\right]\\
&=& \lim_{n\rightarrow \infty} n\E\left[\^{M}_{\psi}Q_c\^{M}_{\psi}P_c\right]\\
&=& \E[ Z Z^T P_c \| Q_c Z\|^2 \psi(Y)^2],
\end{eqnarray*}
which complete the proof of the proposition.
\end{proof}

Proposition $\ref{calvar}$ provides us the expression of the quantity to minimize with respect to the function $\psi$. The next lines are attached to find $\psi$ such that $\E[W_{\psi}]$ is minimal. This informal calculation leads us to a fixed point equation whose solution is expected to be the minimum of $\E[W_{\psi}]$. Thanks to proposition \ref{calvar} the quantity to minimize can be written as
\begin{equation*}
\E[W_{\psi}]=\tr( \E[ZZ^TP_c \|Q_cZ\|^2 \psi(Y)^2] (P_cM_{\psi}-I\lambda_{\psi}^{\ast})^{-2} ),
\end{equation*}
or with the notations $A=ZZ^TP_c \|Q_cZ\|^2$ and $B=P_cZZ^T- \frac{\|Q_cZ\|^2}{p-d} I$,
\begin{equation*}
\E[W_{\psi}]=\tr\left( \E[A\psi(Y)^2]\ \E[B\psi(Y)]^{-2} \right).
\end{equation*}
Thus we are looking for $\psi$ such that for every bounded measurable function $\delta$,
\begin{equation*}
 \left. \parti{t} \E[W_{\psi+t\delta}]\right|_{t=0}=0,
\end{equation*}
or equivalently,  
\begin{multline*}
\E\bigg[ 2\tr\left(A\delta \psi \E[B\psi]^{-2} \right)\\ -\tr\left(\E[A\psi^{2}] \E[B \psi]^{-1} \{ B \delta  \E[B \psi]^{-1} 
+\E[B \psi]^{-1}  B \delta \} \E[B \psi]^{-1}\right) \bigg ]=0,
\end{multline*}
where $\delta$ and $\psi$ stand for $\delta(Y)$ and $\psi(Y)$. Define the functions $A(Y)= \E[A|Y]$ and $B(Y)=\E[B|Y]$. Since the previous equation should be true for any $Y$-measurable random variable $\delta(Y)$, we derive
\begin{multline*}
2\tr\left(A(Y) \psi(Y) \E[B\psi]^{-2} \right)\\-\tr\left(\E[A\psi^{2}] \E[B \psi]^{-1} \{ B(Y)   \E[B \psi]^{-1} 
+ \E[B \psi]^{-1}  B(Y) \}  \E[B \psi]^{-1} \right)=0\ \ \ \ \text{a.s.,}
\end{multline*}
which leads to the implicit equation 
\begin{equation}\label{psi}
\psi(y)=\frac{    \tr\left(\E[B \psi]^{-1} \E[A\psi^{2}]\E[B \psi]^{-1}\{\E[B \psi]^{-1} B(y)   +   B(y) \E[B \psi]^{-1}\}  \right)   }{  2\tr(  A(y) \E[B\psi]^{-2}  )  }.
\end{equation}
This solution describes the optimal $\psi$ function to perform TF2. To find this $\psi$, we propose an iteration of the point fixed equation (\ref{psi}). Before we state a more accurate algorithm to compute TF2, we set out a new way to express (\ref{psi}). As we highlighted at the beginning of the section, based on Theorem \ref{apsi} we suppose that $\psi$ is such that we can reach the eigenvectors $\eta_{\psi}=(\eta_1,...,\eta_d)\in \R^{p\times d}$ of $M_{\psi}$ that are in $E_c$. Therefore we can write $P_c=\eta_{\psi} \eta_{\psi}^T$ and by definition of $\eta_{\psi}$, we have 
\begin{equation}\label{vp}
\E[B\psi(Y)]^{-1} \eta_{\psi} = \eta_{\psi} D_{\psi},
\end{equation}
where $D_{\psi}=\diag_k(\lambda_{\psi}(\eta_k)-\lambda_{\psi}^{\ast})^{-1}$. Besides, a simple use of the linearity condition provides that $\E[\eta^T ZZ^T|Y]=\E[\eta^TZZ^TP_c|Y]$ for every $\eta\in E_c$. Consequently, we derive that
\begin{equation}\label{lcond}
\eta_{\psi}^T B(y) = \eta_{\psi}^T B(y)P_c.
\end{equation}
Then with the introduced notations and using (\ref{vp}) and (\ref{lcond}), we obtain this other formulation of (\ref{psi}), 
\begin{equation*}
\psi(y)=\frac{ \tr\left(D_{\psi} A_{\psi}D_{\psi} \{D_{\psi} \widetilde{B}(y)   +   \widetilde{B}(y)D_{\psi}\}  \right)   }{  2\tr\left(  \widetilde{A}(y) D_{\psi}^{2}  \right)  },
\end{equation*}
where 
\begin{equation*}
A_{\psi}=\E\left[\eta_{\psi}^TZZ^T\eta_{\psi}\|Q_cZ\|^2\psi(Y)^2\right],\ \widetilde{A}(Y)=\eta_{\psi}^T A(y) \eta_{\psi},\ \widetilde{B}(y)=\eta_{\psi}^T B(y) \eta_{\psi},
\end{equation*}
are $d\times d$ matrices. Using the symmetry of the matrices $A_{\psi}$ and $\widetilde{B}(y)$, and some well-known properties of the trace, we obtain
\begin{equation}\label{psi2}
\psi(y)=\frac{ \tr\left(D_{\psi} A_{\psi}D_{\psi} \widetilde{B}(y)D_{\psi}  \right)   }{  \tr\left( \widetilde{A}(y) D_{\psi}^{2}  \right)  }.
\end{equation}
Since $\widetilde{A}$ and $\widetilde{B}$ are unknown function, we use a slicing approximation and it gives
\begin{equation}\label{aproxpsi2}
\psi(y)=\sum_h \frac{ \tr\left(D_{\psi} A_{\psi}D_{\psi}  \widetilde{B}_h D_{\psi}  \right)   }{  \tr\left(  \widetilde{A}_h   D_{\psi}^{2}  \right)  }\mathds{1}_{\{y\in I(h)\}},
\end{equation}
where $\widetilde{A}_h=\E[\widetilde{A}(Y)\mathds{1}_{\{y\in I(h)\}}]$ and $\widetilde{B}_h=\E[\widetilde{B}(Y) \mathds{1}_{\{y\in I(h)\}}]$. Now we set out the TF2 method based on the family of indicator functions. In practice, the fixed point equation (\ref{psi2}) gives better results than (\ref{psi}), therefore we use (\ref{aproxpsi2}) to compute TF2. We propose the following algorithm that describes the iteration needed to implement our method. To be more comprehensive, we based the algorithm on the weights $\alpha_h$ instead of the function $\psi_h(y)=\sum_h\alpha_h\mathds 1 _{\{y\in I(h)\}}$. Besides $\^A_{\^\psi}$ and $\^D_{\^\psi}$ are noted $\^A$ and $\^D$, and we will need
\begin{equation*}
M_h=\E[ZZ^T\mathds{1}_{\{Y\in I(h)\}}]\quad \text{and} \quad \lambda_h = \E\left[\frac {\|Q_cZ\|^2}{p-d}\mathds{1}_{\{Y\in I(h)\}}\right].
\end{equation*}
Because $\lambda_h$ is the eigenvalue associated to the space $E_c^{\perp}$, we estimate it the following way, supposing that $\dim(E_c)<\dim(E_c^{\perp})$.

\textbf{TF2 Algorithm:}
\begin{enumerate}
\setcounter{enumi}{-1}
\item Each $I(h)$ contains $\frac n H$ observations. Compute
\begin{equation*}
\^M_h = \frac 1 n\sum_{i=1}^n Z_i Z_i^T \mathds{1}_{\{Y_i\in 
I(h)\}},\quad  
\^\lambda_h = \text{median}(\lambda \in \text{spectrum}(\^M_h)),
\end{equation*}
and initialize $\^\alpha_{h}\sim \mathcal{U}[0,1]$ for every $h=1,...,H$.

\item Identify the eigenvectors $\^\eta=(\^\eta_1,...,\^\eta_d)\in E_c$ of $\^M=\sum_h 
\^\alpha_h\^M_h$. 

\item\label{st2} Derive $\^D=\diag_k(\^\lambda_{\^\psi}(\^\eta_k) - \^\lambda_{\^\psi}^{\ast})^{-1}$, $\^Q_c=I-\^\eta\^\eta^T$ and 
\begin{equation*}
\^A=\sum_h \^\alpha_h \^\eta^T\^A_h \^\eta,\quad \text{with}\ \^A_h = 
 \frac 1 n \sum_{i=1}^n Z_i Z_i^T \|\^Q_c Z_i\|^2 \mathds{1}_{\{Y_i\in I(h)\}}.
\end{equation*}

\item Compute\begin{equation*}
\^\alpha_h= \frac{ \tr\left(\^D^2 \^A\^D\ \ (\^\eta^T\^M_h 
\^\eta -\^\lambda_h I) \right)   }
{  \tr\left(\^D^{2}\ \ \^\eta^T \^A_h  \^\eta  \right)  }.
\end{equation*}
\end{enumerate}
Repeat the last three steps until the convergence is achieved. The resulting function $\^\psi$ is an estimate of the solution of the fixed point equation. Finaly the set of vectors $\^\eta$ form an estimated basis of the standardized CS. The space generated by $\^\Sigma^{-\frac{1}{2}}\^\eta$ provides an estimation of the CS by TF2.

\begin{remark}\label{practicetf2}
A crucial point need to deserve our attention. It concerns the way we identify the eigenvectors of $M_{\psi}$ 
that belong to $E_c$ and a fortiori their associated eigenvalues. It intervenes at each iteration of our algorithm to estimate $D_{\psi}$ and $\eta_{\psi}$. The theoretical background of the TF2 method advocates for an identification process based on the eigenvalues more than the eigenvectors. Indeed, as it is pointed out at the end of section \ref{s3} the eigenvalues of $M_{\psi}$ associated to eigenvectors of $E_c^{\perp}$ are all equal. We tried to base an algorithm on this fact but it appeared that it was not robust to small samples. So that we choose to develop another one which takes into account the nature of the eigenvectors of $M_{\psi}$. Let $\eta$ be an eigenvector of $M_{\psi}$, we based a new identification process on the dependence between $(\eta^T Z)$ and $Y$. We propose to compare the Pearson's chi-square statistic of the test of independence between $(\eta^TZ)$ and $Y$. Therefore, for each eigenvector we divide the range of $(\eta^TZ)$ into $H$ slices noted $J(h)$ and we calculate
\begin{equation}\label{ident}
S(\eta)= \sum_{h,h'}\frac{ \left(p_{hh'}-\overline{p_{hh'}}^h\ \overline{ p_{hh'}}^{h'}\right)^2   }{\overline{p_{hh'}}^h\ \overline{ p_{hh'}}^{h'}} 
\end{equation}
where $p_{h,h'}=\frac{1}{n} \sum_{i=1}^n \mathds{1}_{\{Y_i \in I(h)\}} \mathds{1}_{\{(\eta^TZ_i)\in J(h')\}}$. Then the $d$ eigenvectors of $M_{\psi}$ associated to the largest values of $S$ are identified as being in $E_c$. As a consequence, at step \ref{st2} of the TF2 Algorithm, the $\^\lambda_{\^\psi}(\^\eta_k)$'s are the eigenvalues of $\^M$ associated to the eigenvectors $\^\eta_k$'s with the $d$ largest values of $S$, $\lambda_{\^\psi}^{\ast}$ is the median over the other eigenvalues. In the next section dedicated to simulations, criterion (\ref{ident}) has been used to compute TF2. 
\end{remark}

\section{Simulations}\label{s7}
In this section, we first compare the performance of the order $1$ test function variance minimization with the performance of the SIR estimator. Then, we compare some order $2$ methods through pathological models for order $1$ methods (see example \ref{patho}). To measure the performance of a method we evaluate the error between the CS and its estimate with the following distance: for two subspace $E_1$ and $E_2$, if $P_1$ and $P_2$ are their respective orthogonal projection, the distance between $E_1$ and $E_2$ is  
\begin{equation}\label{crit}
\text{Dist}(E_1,E_2)=\|P_1-P_2\|^2,
\end{equation}
where $\|\cdot\|$ stands for the Frobenius norm.

Besides, since TF1 and TF2 are performed with the family of indicator functions, we have to discretize the response into $H$ slices. The slices are built in such a way that each slice contains the same number of observations.   

\subsection{Order $1$ test function}
Let us consider the case where the predictors have a gaussian distribution. Clearly $P_cZ$ and $Q_cZ$ are two independent random vectors and then $\E[\|Q_cZ\|^2 |Y]=\E[\E[\|Q_cZ\|^2|P_c Z]|Y ]=p-d$. Therefore $\spann (M_{TF1})=\spann (M_{SIR})$ and TF1 provides exactly the same estimator as SIR. Simulations made in this case highlight the similarity between both methods and are not presented here.

Consequently, to point out the differences between these two methods, we generate non-gaussian predictors. Taking $X=\rho U$ where $U$ is a uniformly distributed vector on the unit sphere of $\R^p$ independent of $\rho$, which is a real random variable. A first point is that $X$ has a spherical distribution. Moreover, we take
\begin{equation}\label{rho}
\rho =  \epsilon\ |10+0.05 W_1|\ +\ (1-\epsilon)\ |30+0.05 W_2|, 
\end{equation}
with $W_1\sim \mathcal{N} (0,1)$, $W_2\sim\mathcal{N}(0,1)$ and $\epsilon \sim \mathcal{B}(\frac 1 2) $. We performed SIR and TF1 on the following two models. Model I is derived from \citet*{li1991} and considered in many articles on the subject,
\begin{eqnarray*}
\text{Model I:} \quad\quad Y &=& \frac {X_1}{ 0.5+(X_2+1.5)^2} + 0.5\varepsilon\\
\text{Model II:} \quad\quad Y &=& \text{sign}(X_2)|X_1/2+5|+0.5\varepsilon,
\end{eqnarray*}
where $\varepsilon \sim \mathcal{N}(0,1)$. We have to standardize $X$ into $Z$ to compute TF1. Clearly, the variance of $X$ is proportional to the identity matrix, then the standardized directions are the same than the non-standardized one. For models I and II, directions to estimate are $(1,0,...,0)^T$ and $(0,1,0,...,0)^T$. 

To be more comprehensive, for each model we compute both methods with some different configuration of the parameters $(n,p,H)$ which are taken as $(100,6,5)$, $(500,10,10)$ and $(1000,20,20)$. For each configuration, we perform $100$ simulated random samples. Some boxplots of the distances measured between the estimated and the true CS are presented in figure \ref{fig1}. 

\begin{figure}
\includegraphics[width=6.2cm,height=6cm]{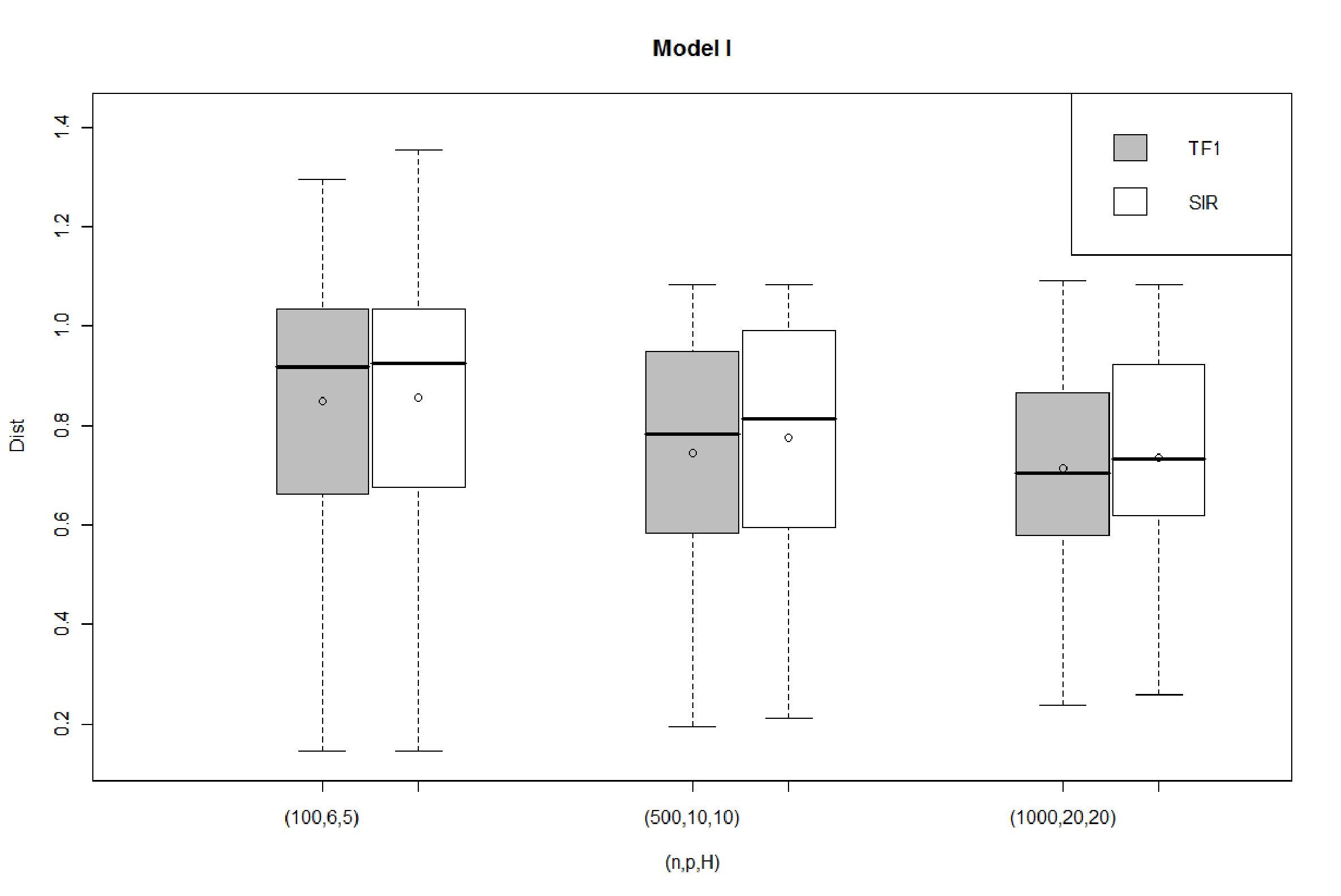} 
\includegraphics[width=6.2cm,height=6cm]{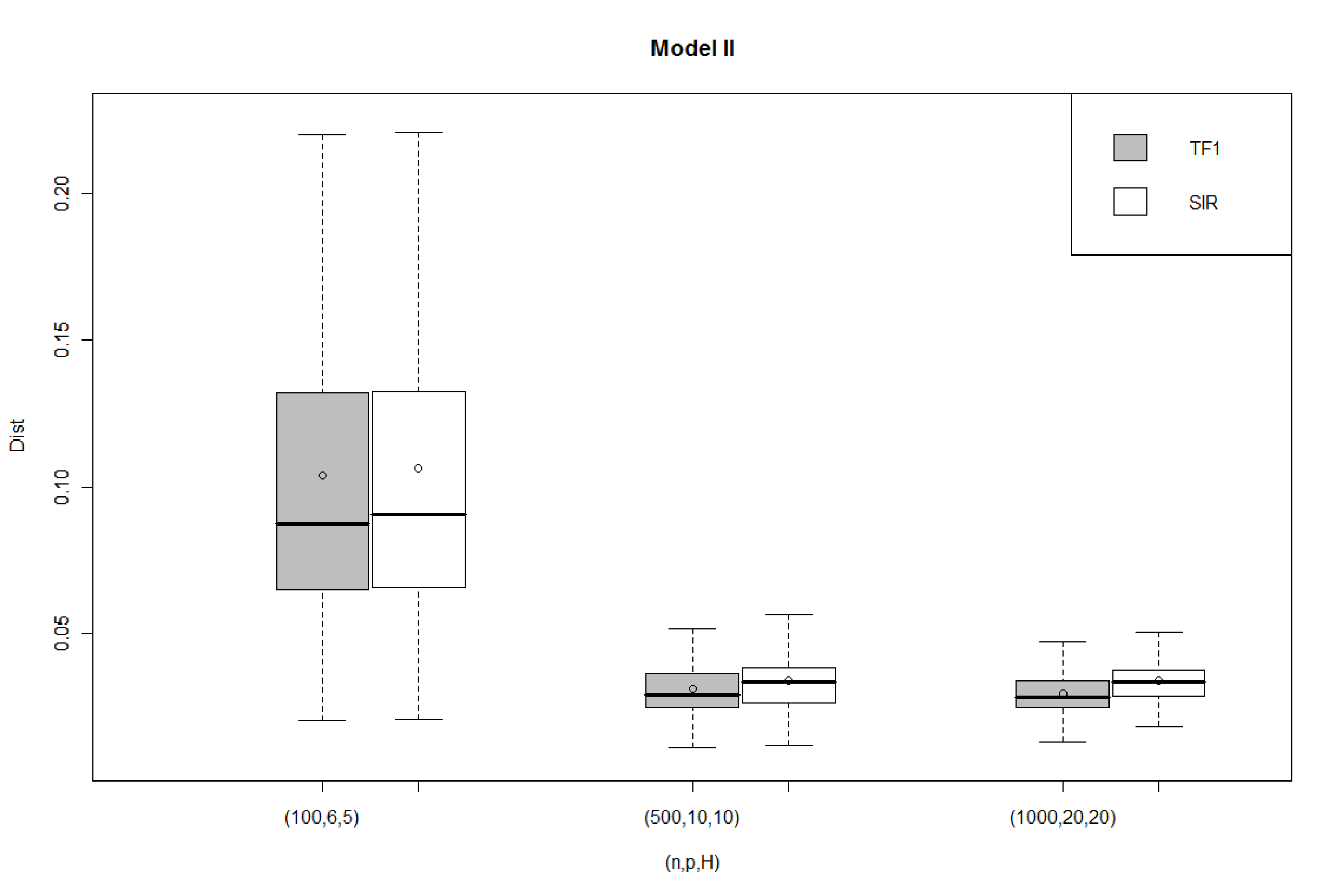} 
\caption{\label{fig1}Comparison of TF1 and SIR when $X$ has a spherical distribution.}
\end{figure}

For each model and in all the parameters configurations, TF1 performs better than SIR. Model II reflects a suitable situation for order $1$ methods because its regression function is not symmetric with respect to any of its coordinates. As a consequence the measured errors are quite small for both methods. Model I indicates a more difficult situation. Indeed, because the standard error of $X_2$ is near $16\gg1.5$, the regression function associated to model I is almost symmetric with respect to its second coordinate (see model \ref{patho}). It appears that both methods have difficulties in finding this coordinate. Figure \ref{fig1} shows that in each situation the difference between the performance of both methods increases with the sample size. Nevertheless, because of the high level of similarity between the theoretical background of these two methods, the distances presented are really close. Especially for $n=100$, where the improvement of TF1 is not really significant.  

To reach a point of view developed in the simulation study of \citet*{cook2005}, we are interested in the link between the variation of $\var(Z|Y)$ and the performance of the presented method. First, according to equation (\ref{poids}), the variation of the random variable $\E[\|Q_cZ\|^2|Y]$ is essential in studying the differences between SIR and TF1. Indeed if this one is a constant, then $d_h=\E[\|Q_cZ\|^2 \mathds{1}_{\{Y\in I(h)\}}]=(p-d) p_h$ and TF1 is the same method than SIR. Consequently, SIR estimates near optimal with respect to criterion (\ref{crit}) when the variations in $\E[\|Q_cZ\|^2|Y]$ are near $0$. Besides, if this random variable is nonconstant then also the $d_h$ and the differences between both methods are highlighted. Secondly, we can notice that $\E[\|Q_cZ\|^2|Y]$ and $\var(Z|Y)$ are strongly linked. Thanks to the well-known variance decomposition formula, we have
\begin{equation*}
\var(Z|Y)=\E[\var(Z|P_cZ)|Y]+\var(\E[Z|P_cZ]|Y),
\end{equation*}
and using the linearity condition, we obtain that
\begin{equation*}
\tr(\var(Z|Y))=\E[\|Q_cZ\|^2|Y] + \tr(\var(P_cZ|Y)).
\end{equation*}
Thus, as it was the case to distinguish IRE from SIR, it seems that the variations of $\var(Z|Y)$ has an important role to differentiate TF1 from the SIR.

As it has been studied in some recent papers like \citet*{dong2009} and \citet*{dong2010}, we introduce nonlinearity in the distribution of the predictors. Although it does not correspond to the set of assumptions required in SIR and TF1 theoretical background, it is interesting to provide the following results as an indicator of the robustness of each method. Here, predictors are generated as previously but we change $X_1$ and $X_2$ as follows,
\begin{eqnarray*}
X_1 &=& 0.2 X_3+0.2 (X_4+10)^2+0.2u,\\
X_2 &=& 0.1+0.1(X_3+X_4)+0.3X_3^2+0.2u,
\end{eqnarray*} 
where $u\sim \mathcal{N}(0,1)$. Model III is the same than model I but with the above predictors distribution. We provide boxplots of the estimation error of the $100$ simulated random sample in figure \ref{fig2}. 

\begin{figure}
\includegraphics[width=7.5cm,height=6cm]{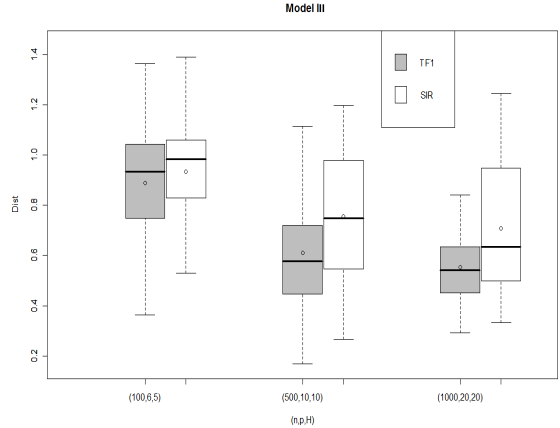} 
\caption{\label{fig2}Comparison	of TF1 and SIR when there is nonlinearity between the predictors.}
\end{figure}

In this case, figure \ref{fig2} shows a large difference between the estimation error of SIR and TF1. TF1 performed better in each case and the difference between both methods increases as $n$ is large.

\subsection{Order $2$ test function. Symmetric model}
We now compare several well-known order $2$ dimension reduction methods with TF2. Order $2$ methods we have computed include SAVE, pHd, SIR-II and DR. For the models we consider here, pHd and SIR-II do not work as well as the others. Therefore we focus on a comparison between SAVE, DR and TF2.

TF2 estimation is not as close to DR and SAVE than the order TF1 estimation is closed to SIR. The following simulations highlight this fact and as a consequence we begin this section by providing the results obtained with gaussian predictors. We considerer the following three regression models, note that model V is derived from \citet*{li2007},
\begin{eqnarray*}
\text{Model IV:}\quad\quad Y &=& 4\tanh\left(\frac{|X_1|}{ 2}\right) + 0.5\varepsilon\\
\text{Model V:}\quad\quad  Y &=& 0.4X_1^2+\sqrt {|X_2|}  +0.2\varepsilon\\
\text{Model VI:} \quad\quad  Y &=&  1.5 X_1X_2\ \varepsilon
\end{eqnarray*}
with $\varepsilon \sim \mathcal{N}(0,1)$ and $X\sim \mathcal{N}(0,I_p)$. The CS of model IV is spanned by the direction $(1,0,...,0)$, whereas in Model V and VI, it is a two dimensional subspace generated by $(1,0,...,0)$ and $(0,1,0,...,0)$. As the simulations for the order $1$, we consider different parameter configurations where each of the presented method is in a convenient situation. We simulate SAVE, DR and TF2 with $(n,p,H)$ equal to $(100,6,5)$, $(500,10,5)$ and $(1000,20,10)$. For each configuration, $100$ simulated random samples have been performed and the resulting boxplots with their averages are presented in figure \ref{fig3}.
\begin{figure}
\includegraphics[width=4.1cm,height=6cm]{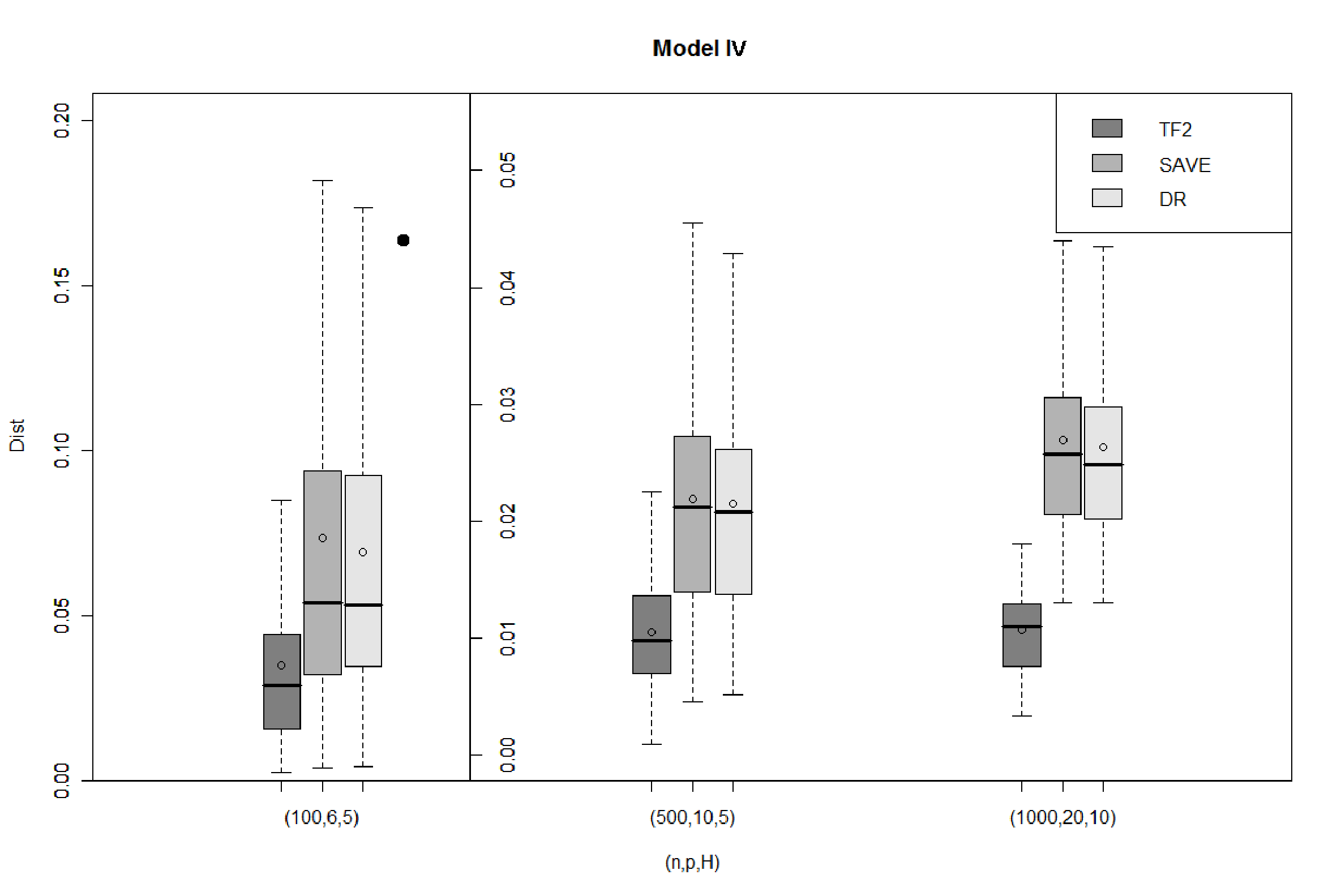} 
\includegraphics[width=4.1cm,height=6cm]{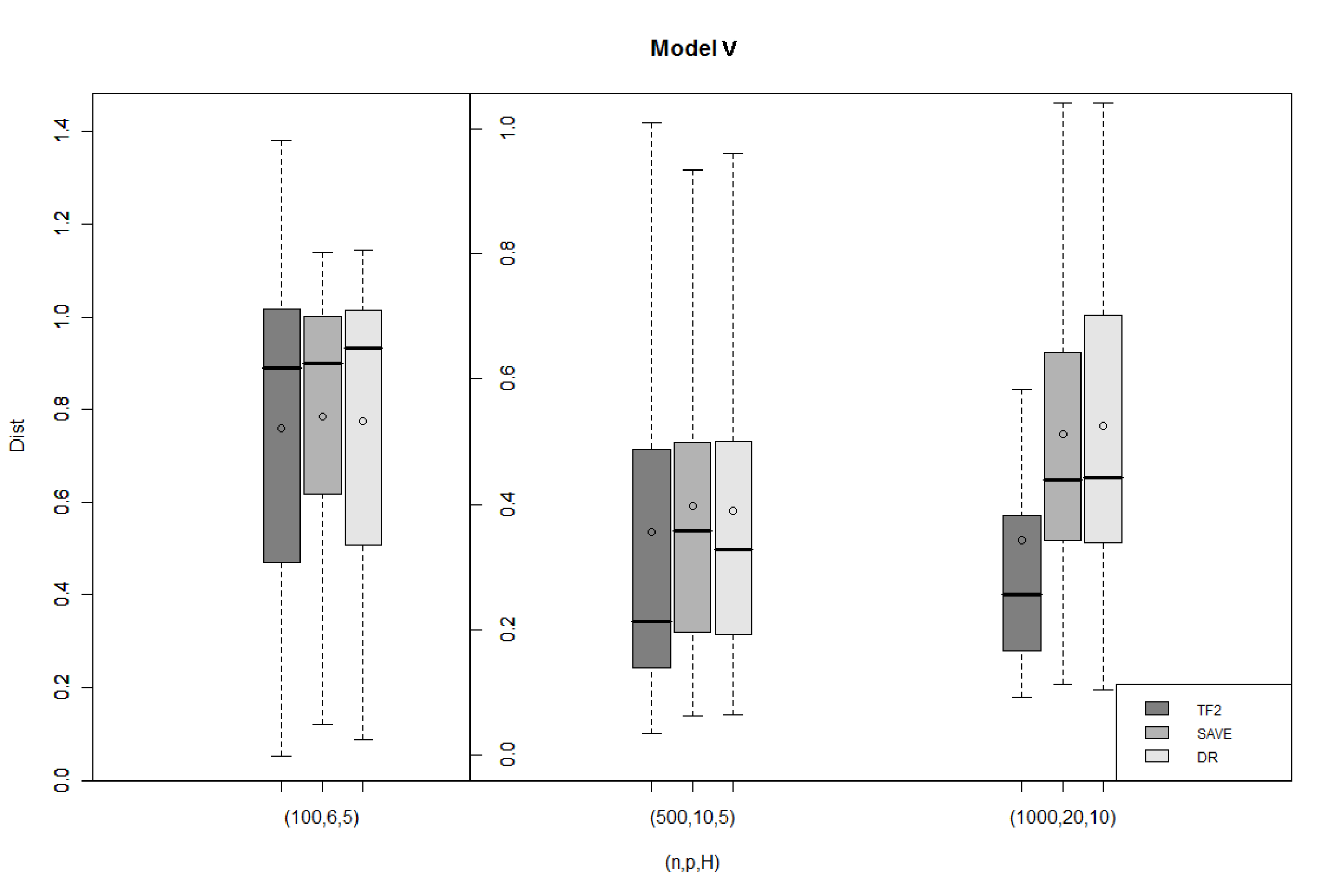} 
\includegraphics[width=4.1cm,height=6cm]{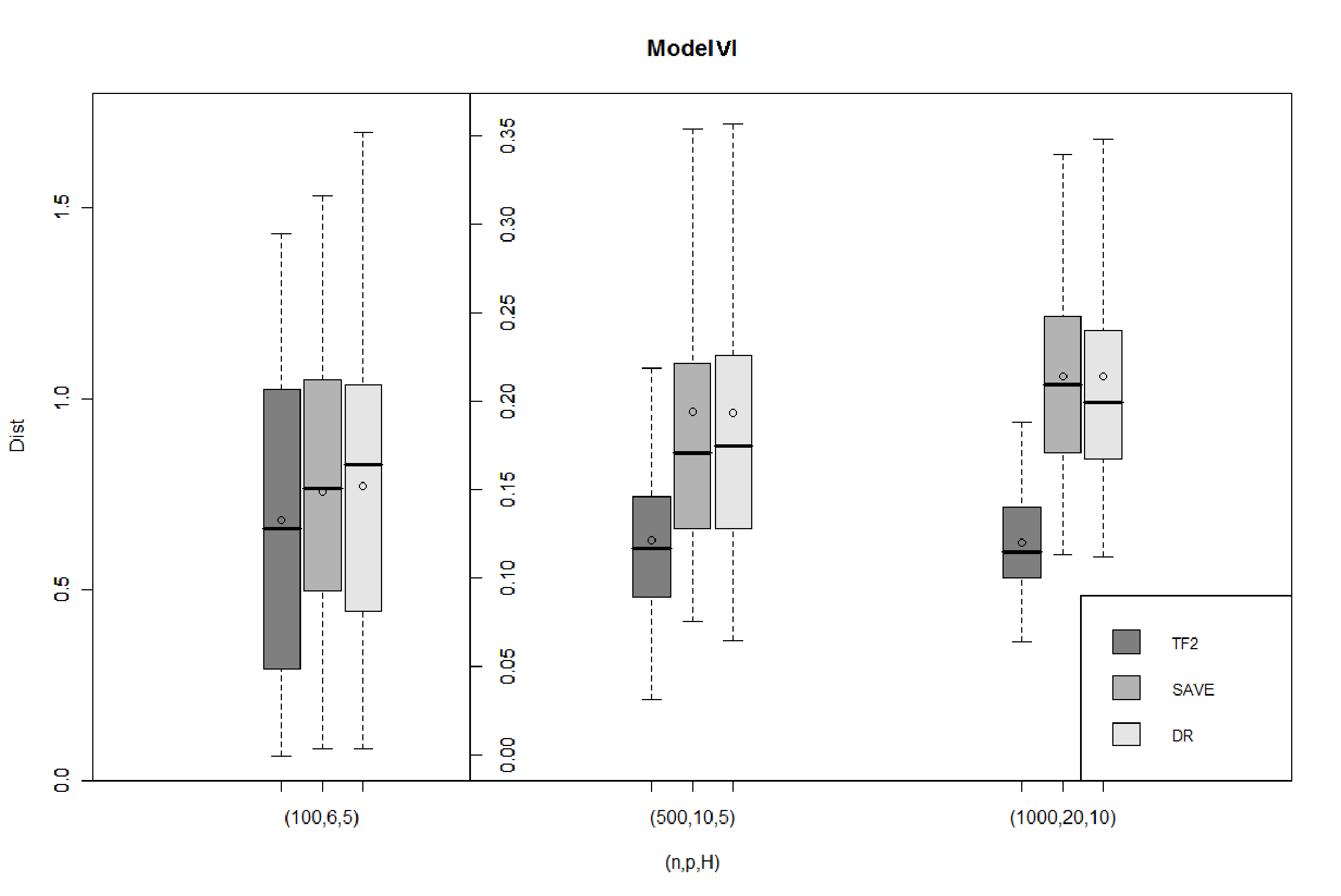} 
\caption{\label{fig3}Comparison of TF2, SAVE and DR	when $X$ has a gaussian distribution.}
\end{figure}

For all the selected models, TF2 perform better than DR and SAVE. The most significant improvement happens for model IV in which our method perform better than the others around $90\%$ of the time in each $(n,p,H)$ configuration. Note that for $n=100$, $500$, the mean of the TF2 is two times smaller than the mean of DR or SAVE. For $n=1000$ this factor goes to three. The results of the simulation for model VI are really close from model IV. Model V is a more complicated one for each method. Moreover, we have to wait $n=1000$ to remark substantial differences in the distribution of the criterion. In every model, the criterion mean of TF2 is the smallest and as $n$ is large, as the improvement of TF2 looks substantial. Besides, it is clear that for the selected models, SAVE and DR perform in a similar way.  

\begin{remark}
For our study and the development of TF2, model V was a really interesting one. In figure \ref{fig3}, for $n=100$ the mean is less than the median, and it is no longer the case for $n$ larger than $100$. This marked change in the boxplots is explained by the presence of small outliers in the first situation and large outliers in the second one. Indeed as $n$ is large, TF2 performs better but however, the mean is shifted by the presence of outliers that reflects uncommon difficult situations. As it is explain in section \ref{s52}, TF2 relies on the way to identify eigenvectors of $M_{\psi}$ that belong to $E_c$. To make that possible, a test of independence between the response and the projected predictors is conducted. Outliers of model V for $n$ equal to $500$ and $1000$ are the consequence of a bad eigenvector choice realized by this test. When $n$ is sufficiently large this no longer occurs. When the TF2 algorithm is iterated a larger number of times, it happens only very few times.       
\end{remark}

To conclude this simulation section we present the results obtained with spherical predictors. Here, $X$ is generated with the equation $X=\rho U$ where $U$ is a uniformly distributed vector on the unit sphere of $\R^p$, independent of $\rho$ defined by equation (\ref{rho}). Again we study the model IV and also the following ones,
\begin{eqnarray*}
\text{Model VII:} \quad\quad Y &=& |X_1|+\left(\frac {X_2}{ 4}\right)^2 +0.5\varepsilon\\
\text{Model VIb:} \quad\quad Y &=&  X_1X_2\ \varepsilon
\end{eqnarray*}
where $\varepsilon\sim \mathcal{N} (0,1)$. Model VI has been changed to reduce the signal to noise ratio. The directions to estimate, the parameter configuration and the number of simulated random sample are the same than in the Gaussian case studied previously. Boxplots and their associated averages are presented in figure \ref{fig4}. 
\begin{figure}
\includegraphics[width=4.1cm,height=6cm]{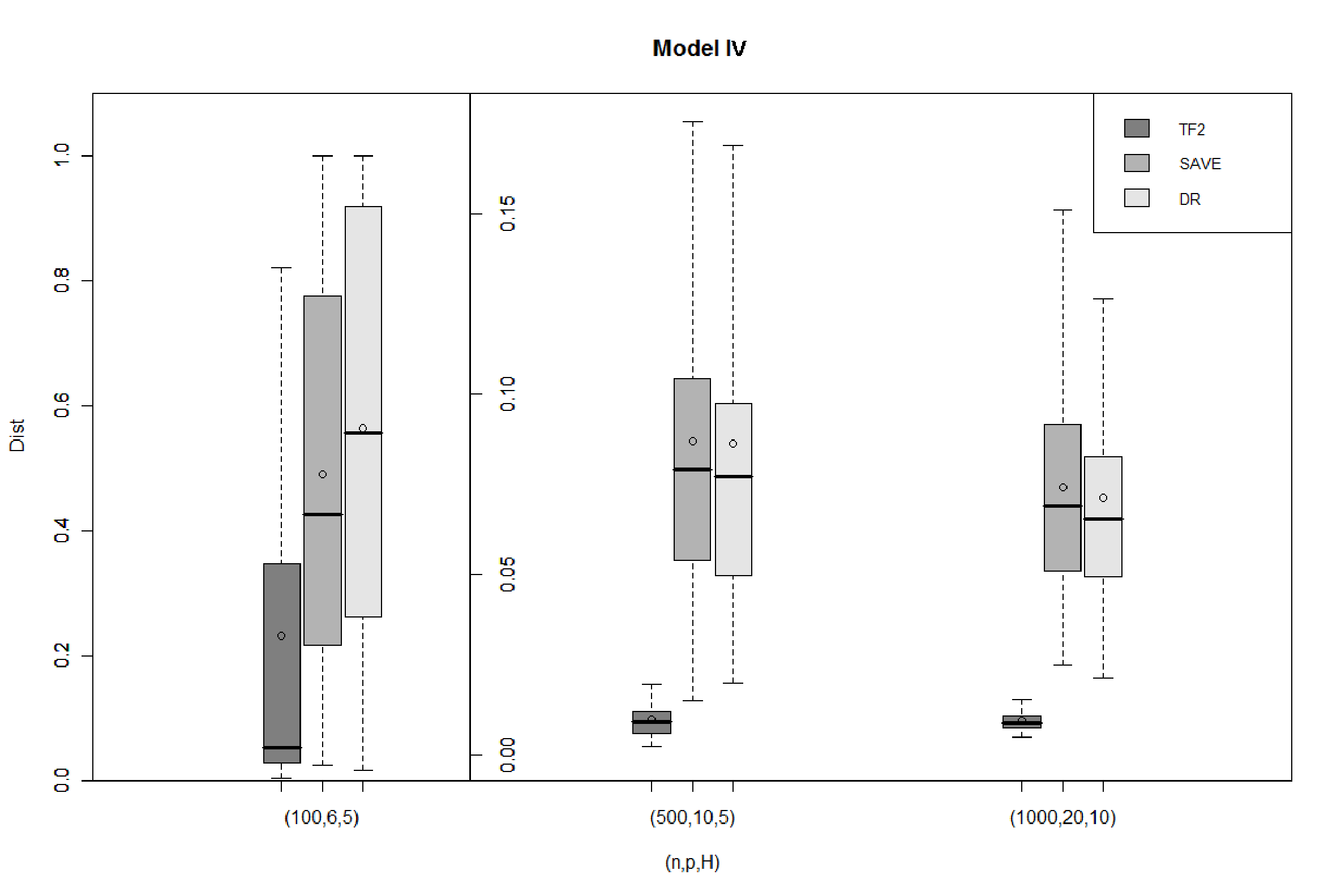} 
\includegraphics[width=4.1cm,height=6cm]{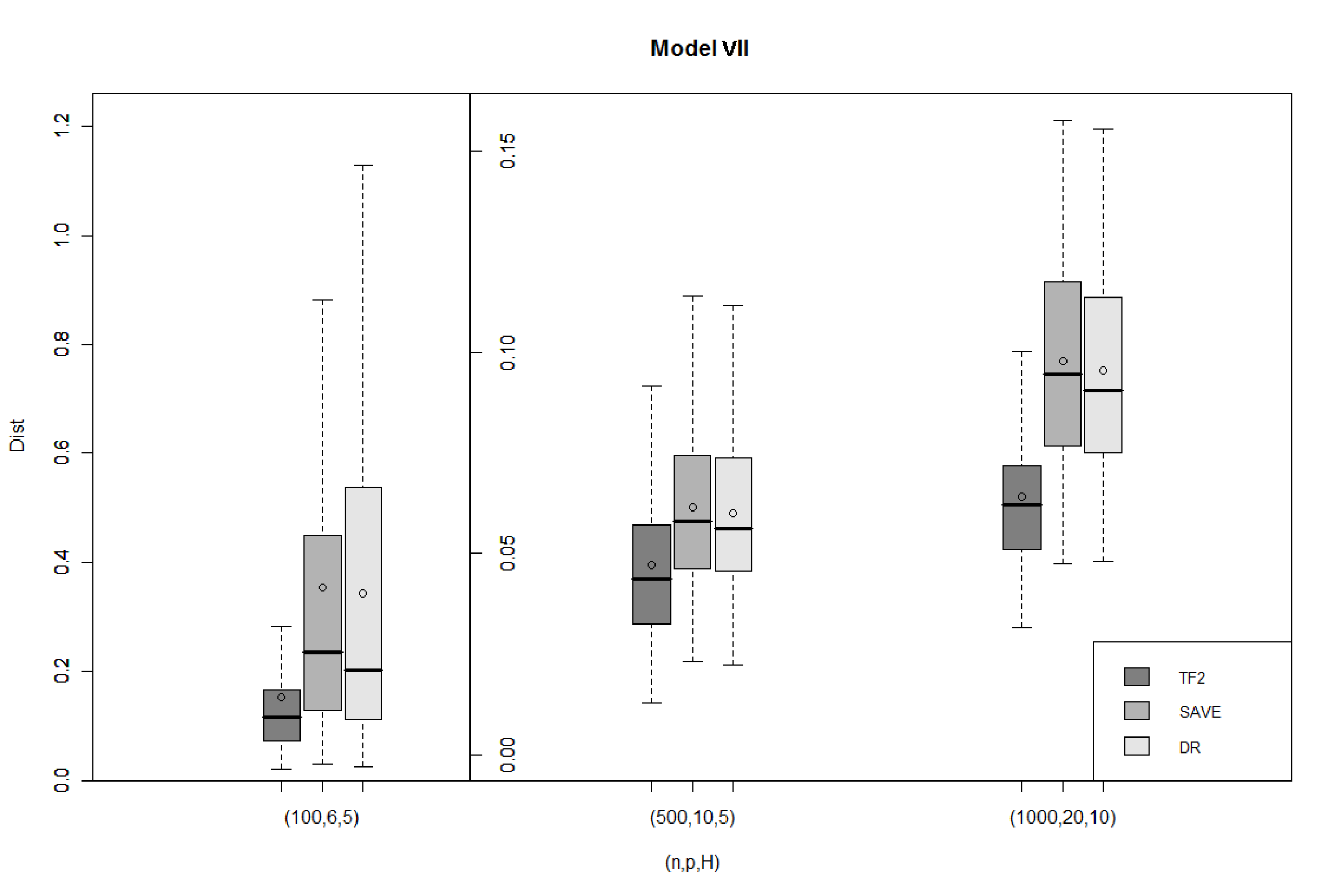} 
\includegraphics[width=4.1cm,height=6cm]{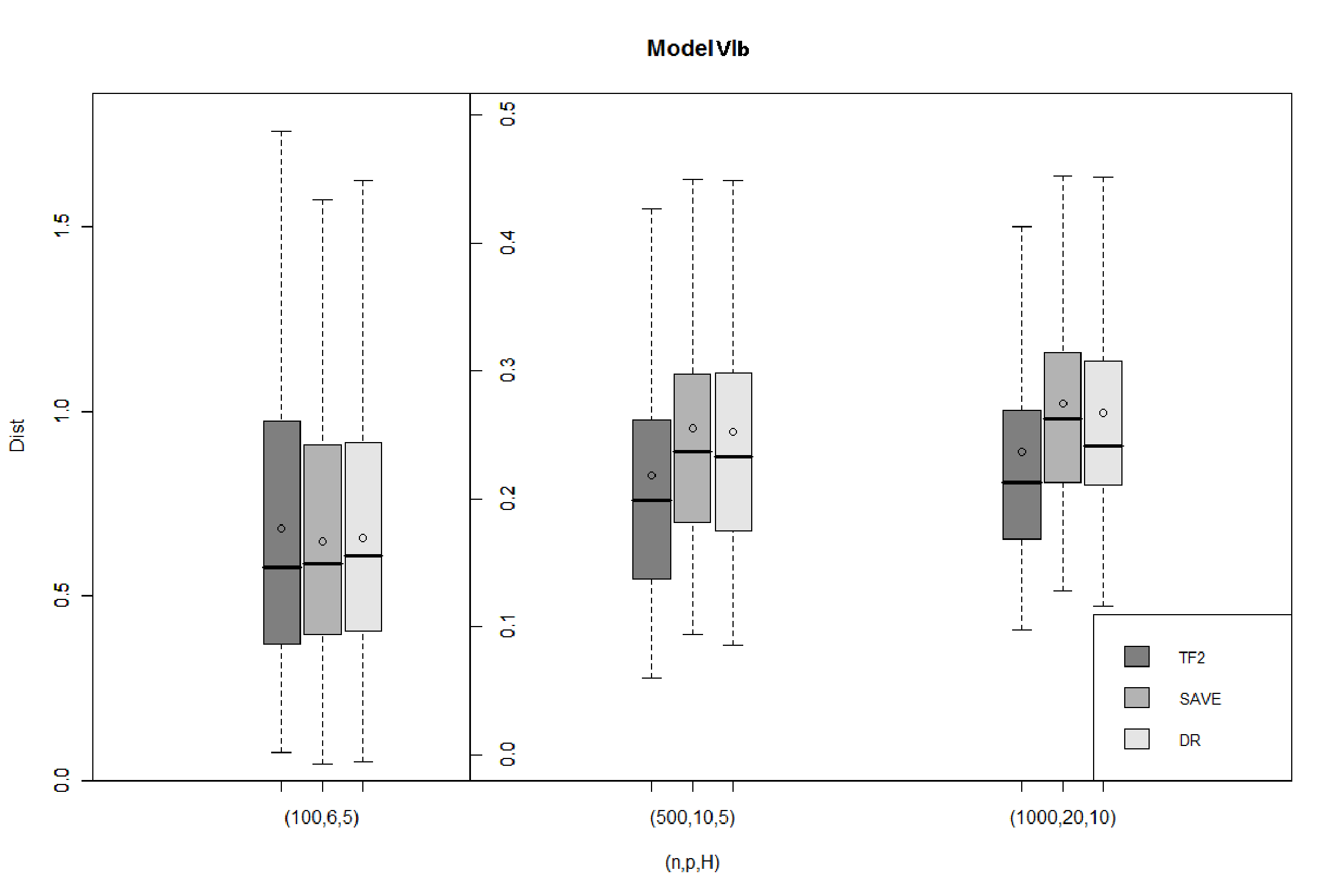} 
\caption{\label{fig4}Comparison of TF2, SAVE and DR	when $X$ has a spherical distribution.}
\end{figure}

Model I still reflects the most important improvement of TF2 with respect to SAVE and DR. When $n$ is large, it performs around height times better than the other. In model VIb, TF2 estimation deteriorates by changing distribution of the predictors from gaussian into spherical. Finally, model VII provides a standard new situation where the improvement of TF2 is highly significant.

\section{Concluding remarks}
This article introduces the basis of a new methodology about SDR. Although the theoretical background of TF1 and TF2 is quite the same than SDR methods, the methods proposed work under weaker conditions than the ones of the literature. Moreover, the resulting estimation methods are not at all the same. Indeed, the introduction of some transformation of the response was the original idea of this work and has led us to some new way of investigation in SDR. A surprising point was the similarity between SIR and the TF1 variance minimization. For TF2, the simulation study underlines its high accuracy over other order $2$ methods and legitimates the use of TF. However, the framework develop here is not yet completed. 

First, the estimation of the dimension of the CS has been avoided in the present work. Prospects can be find in the Pearson's chi-square statistic used to select a basis of the CS: this statistic could also be employed to estimate the dimension of the CS. Simulations about such a dimension estimation method provided until now some good results. Moreover, an idea which is still under development, is to incorporate such a test in the TF2 algorithm.

Secondly, TF offers a lot of different methods deriving from the choice of a family of functions that separates the points (see Remark \ref{coudene} and Corollary \ref{sumpsi}). Here we attached to study TF with indicators. The Fourier basis or a polynomial family could also be considered to derive some new methods. Besides for TF1 and TF2, a smooth kernel estimation of the function $\psi$ may lead to better convergence rates.   

Finally, we have some few words about a set of methods called hybrid. Some regression function has different kind of components. Consequently, in many cases a particular method would provide a good estimate of some components but another one would be needed to infer about the remaining components. This clearly argues for methods that are a mixing of the existing ones. This kind of methods are usually called hybrid method, they can be summarized by the equation 
\begin{eqnarray*}
M=\alpha M_1+(1-\alpha)M_2,
\end{eqnarray*}
where $M_1$ and $M_2$ are the associated matrix of two different methods. A spectral decomposition of $M$ gives an hybrid estimation of the CS. This kind of consideration were recommended by \citet*{saracco2003}, and \citet*{ye2003} proposed a bootstrap method to select the parameter $\alpha$. This includes the combination of SIR and SAVE, SIR and pHd, and SIR and SIR-II. Besides, it is commonly known that
\begin{equation*}
M_{SAVE} = M_{SIR}^2+M_{SIR-II},
\end{equation*}
and that
\begin{equation*}
M_{DR}=\E[\E[(ZZ^T|Y]-I)^2]+M_{SIR}^2+\tr(M_{SIR})M_{SIR},
\end{equation*} 
making SAVE and DR some combinations of SIR and order $2$ moments based methods. Therefore SAVE and DR do not only involve order $2$ moments of the predictors given the response. Thus it seems more realistic to develop hybrid methods based on TF1 and TF2 matrices and specifically, a choice of the parameter $\alpha$ could be realized by the optimization of a well chosen criterion as it has been done independently in TF1 and TF2. Work along this line is in progress.

\appendix
\section*{appendix}
The following lemma is a simplified version of a result about subspaces of
non-invertible matrices (see \citet*{draisma2006}, proposition 3). 
\begin{lemma}\label{lemmematrix}
Let $M$, $N$ $\in \R^{d\times d}$ and $\alpha_0>0$. If $\forall \alpha\leq \alpha_0$, $\rank(N+\alpha M)\leq \rank(N)$, then
\begin{equation*}
M\ker(N)\subset \im (N).
\end{equation*}
\end{lemma}
\begin{proof}
Denote by $P_{\alpha}$ the characteristic polynomial of $N+\alpha M$ and define $r_{\alpha}=\rank(N+\alpha M)$ and $k_{\alpha}=\dim (\ker(N+\alpha M))=d-r_{\alpha}$. Because of the continuity of the determinant, the coefficients of $P_{\alpha}$ converge to the coefficients of $P_0$, then $P_{\alpha}$ converges uniformly to $P_0$ on every compact. By definition of $k_0$, $P_0$ is such that
\begin{equation*}
P_{0}(x)=x^{k_0}Q_0(x)\quad\text{with }Q_0(0)\neq 0.  
\end{equation*}
Now we use the uniform convergence. First for $\alpha$ small enough we have $P_{\alpha}^{(k_0)}(0)\neq 0$, and this gives the upper bound $k_\alpha\leq k_0$. Using the assumption we obtain $k_0=k_{\alpha}$. Therefore, again from the uniform convergence, for some $\alpha_0$, 
\begin{equation*}
Q_{\alpha}(0)\neq 0, \quad \alpha\leq \alpha_0.
\end{equation*}  
Clearly, there exists a contour $\mathcal{C}$ such that none of the nonzero eigenvalues of $N+\alpha M$ belong to $\mathcal{C}$, $\alpha\leq \alpha_0$. Using the residue Theorem, this allows us to recover the respective projections $\Pi_0$ and $\Pi_{\alpha}$ on the kernel of the matrices $N$ and $N+\alpha M$ in the following way 
\begin{equation*}
\Pi_0=\oint_{\mathcal{C}} (N-zI)^{-1}\dd z,\quad \text{and} \quad \Pi_{\alpha}=\oint_{\mathcal{C}}(N+\alpha M-zI)^{-1}\dd z,
\end{equation*} 
and we can see that
\begin{equation*}
\Pi_0-\Pi_{\alpha} =\alpha \oint_{\mathcal C} (N-zI)^{-1}M(N+\alpha M-zI)^{-1} \dd z.
\end{equation*}
Because as $\alpha$ goes to $0$, none of the eigenvalues of $N$ and $N+\alpha M$ crosses $\mathcal C$, the integral converges and then we derive that $\Pi_{\alpha} \rightarrow \Pi_ 0$ as $\alpha \rightarrow 0$. Besides, we have 
\begin{equation*}
(N+\alpha M) \Pi_{\alpha}=0, \quad \text{and} \quad N\Pi_0=0,
\end{equation*}   
which lead us to $N(\Pi_0-\Pi_{\alpha})=\alpha M \Pi_{\alpha}$, and we obtain 
\begin{equation*}
\im(MP_{\alpha}) \subset \im(N).
\end{equation*}
Using the continuity of $\Pi_{\alpha}$, we conclude the proof.
\end{proof}

\begin{proposition}\label{propmatrix}
Let $\mathcal{M}\subset \R^{d\times d}$ be a linear subspace of noninvertible symmetric matrices. Then 
\begin{equation*}
\exists u\in \R^d,\quad \forall M\in \mathcal{M}, \quad u^T Mu=0.
\end{equation*}
\end{proposition}
\begin{proof}
Since $\mathcal{M}$ is a linear subspace, we can apply Lemma \ref{lemmematrix} with $N$ a matrix of maximal rank in $\mathcal{M}$ and any $M\in \mathcal{M}$. This gives, for every $M$ and every $u\in \ker(N)$,
\begin{equation*}
Mu = N y,
\end{equation*}
with $y\in \R^d$. Because $N$ is symmetric, by multiplying the left-hand side by $u^T$, we obtain $u^TMu=0$.
\end{proof}

\bibliography{bibl}

\begin{thebibliography}{21}
\providecommand{\natexlab}[1]{#1}
\providecommand{\url}[1]{\texttt{#1}}
\expandafter\ifx\csname urlstyle\endcsname\relax
  \providecommand{\doi}[1]{doi: #1}\else
  \providecommand{\doi}{doi: \begingroup \urlstyle{rm}\Url}\fi

\bibitem[Bryc(1995)]{bryc1995}
W{\l}odzimierz Bryc.
\newblock \emph{The normal distribution}, volume 100 of \emph{Lecture Notes in
  Statistics}.
\newblock Springer-Verlag, New York, 1995.
\newblock Characterizations with applications.

\bibitem[Bura(1997)]{bura1997}
Efstathia Bura.
\newblock Dimension reduction via parametric inverse regression.
\newblock In \emph{{$L_1$}-statistical procedures and related topics
  ({N}euchatel, 1997)}, volume~31 of \emph{IMS Lecture Notes Monogr. Ser.},
  pages 215--228. Inst. Math. Statist., Hayward, CA, 1997.

\bibitem[Cook(1998)]{cook1998}
R.~Dennis Cook.
\newblock \emph{Regression graphics}.
\newblock Wiley Series in Probability and Statistics: Probability and
  Statistics. John Wiley \& Sons Inc., New York, 1998.

\bibitem[Cook and Li(2002)]{cook2002}
R.~Dennis Cook and Bing Li.
\newblock Dimension reduction for conditional mean in regression.
\newblock \emph{Ann. Statist.}, 30\penalty0 (2):\penalty0 455--474, 2002.

\bibitem[Cook and Ni(2005)]{cook2005}
R.~Dennis Cook and Liqiang Ni.
\newblock Sufficient dimension reduction via inverse regression: a minimum
  discrepancy approach.
\newblock \emph{J. Amer. Statist. Assoc.}, 100\penalty0 (470):\penalty0
  410--428, 2005.

\bibitem[Cook and Weisberg(1991)]{cook1991}
R.~Dennis Cook and Sanford Weisberg.
\newblock Discussion of ``sliced inverse regression for dimension reduction".
\newblock \emph{J. Amer. Statist. Assoc.}, pages 28--33, 1991.

\bibitem[Coud{\`e}ne(2002)]{coudene2002}
Y.~Coud{\`e}ne.
\newblock Une version mesurable du th\'eor\`eme de {S}tone-{W}eierstrass.
\newblock \emph{Gaz. Math.}, \penalty0 (91):\penalty0 10--17, 2002.

\bibitem[Dalalyan et~al.(2008)Dalalyan, Juditsky, and Spokoiny]{dalalyan2008}
Arnak~S. Dalalyan, Anatoly Juditsky, and Vladimir Spokoiny.
\newblock A new algorithm for estimating the effective dimension-reduction
  subspace.
\newblock \emph{J. Mach. Learn. Res.}, 9:\penalty0 1648--1678, 2008.

\bibitem[Dong and Li(2010)]{dong2010}
Yuexiao Dong and Bing Li.
\newblock Dimension reduction for non-elliptically distributed predictors:
  second-order methods.
\newblock \emph{Biometrika}, 97\penalty0 (2):\penalty0 279--294, 2010.

\bibitem[Draisma(2006)]{draisma2006}
Jan Draisma.
\newblock Small maximal spaces of non-invertible matrices.
\newblock \emph{Bull. London Math. Soc.}, 38\penalty0 (5):\penalty0 764--776,
  2006.

\bibitem[Eaton(1986)]{eaton1986}
Morris~L. Eaton.
\newblock A characterization of spherical distributions.
\newblock \emph{J. Multivariate Anal.}, 20\penalty0 (2):\penalty0 272--276,
  1986.

\bibitem[Gannoun and Saracco(2003)]{saracco2003}
Ali Gannoun and J{\'e}r{$\hat{\text{o}}$}me Saracco.
\newblock An asymptotic theory for {${\rm SIR}_\alpha$} method.
\newblock \emph{Statist. Sinica}, 13\penalty0 (2):\penalty0 297--310, 2003.

\bibitem[Hristache et~al.(2001)Hristache, Juditsky, Polzehl, and
  Spokoiny]{hristache2001}
Marian Hristache, Anatoli Juditsky, J{\"o}rg Polzehl, and Vladimir Spokoiny.
\newblock Structure adaptive approach for dimension reduction.
\newblock \emph{Ann. Statist.}, 29\penalty0 (6):\penalty0 1537--1566, 2001.

\bibitem[Li and Dong(2009)]{dong2009}
Bing Li and Yuexiao Dong.
\newblock Dimension reduction for nonelliptically distributed predictors.
\newblock \emph{Ann. Statist.}, 37\penalty0 (3):\penalty0 1272--1298, 2009.

\bibitem[Li and Wang(2007)]{li2007}
Bing Li and Shaoli Wang.
\newblock On directional regression for dimension reduction.
\newblock \emph{J. Amer. Statist. Assoc.}, 102\penalty0 (479):\penalty0
  997--1008, 2007.

\bibitem[Li et~al.(2005)Li, Zha, and Chiaromonte]{li2005}
Bing Li, Hongyuan Zha, and Francesca Chiaromonte.
\newblock Contour regression: a general approach to dimension reduction.
\newblock \emph{Ann. Statist.}, 33\penalty0 (4):\penalty0 1580--1616, 2005.

\bibitem[Li(1991)]{li1991}
Ker-Chau Li.
\newblock Sliced inverse regression for dimension reduction.
\newblock \emph{J. Amer. Statist. Assoc.}, 86\penalty0 (414):\penalty0
  316--342, 1991.

\bibitem[Li(1992)]{li1992}
Ker-Chau Li.
\newblock On principal {H}essian directions for data visualization and
  dimension reduction: another application of {S}tein's lemma.
\newblock \emph{J. Amer. Statist. Assoc.}, 87\penalty0 (420):\penalty0
  1025--1039, 1992.

\bibitem[Xia et~al.(2002)Xia, Tong, Li, and Zhu]{xia2002}
Yingcun Xia, Howell Tong, W.~K. Li, and Li-Xing Zhu.
\newblock An adaptive estimation of dimension reduction space.
\newblock \emph{J. R. Stat. Soc. Ser. B Stat. Methodol.}, 64\penalty0
  (3):\penalty0 363--410, 2002.

\bibitem[Ye and Weiss(2003)]{ye2003}
Zhishen Ye and Robert~E. Weiss.
\newblock Using the bootstrap to select one of a new class of dimension
  reduction methods.
\newblock \emph{J. Amer. Statist. Assoc.}, 98\penalty0 (464):\penalty0
  968--979, 2003.

\bibitem[Zhu and Fang(1996)]{zhu1996}
Li-Xing Zhu and Kai-Tai Fang.
\newblock Asymptotics for kernel estimate of sliced inverse regression.
\newblock \emph{Ann. Statist.}, 24\penalty0 (3):\penalty0 1053--1068, 1996.

\end{thebibliography}

\end{document}